\documentclass{gIPE2e}

\usepackage{mathtools}
\usepackage[labelformat=simple]{subcaption}

\usepackage[labelsep=period]{caption}
\DeclareCaptionLabelFormat{flabel}{Figure~#1#2}
\usepackage{algpseudocode}
\usepackage{algorithm}

\theoremstyle{plain}
\newtheorem{theorem}{Theorem}[section]

\newtheorem{lemma}[theorem]{Lemma}

\theoremstyle{remark}
\newtheorem{remark}[theorem]{Remark}

\theoremstyle{definition}

\numberwithin{equation}{section} 
\numberwithin{figure}{section}
\numberwithin{table}{section}

\DeclarePairedDelimiter{\norm}{\lVert}{\rVert} 
\DeclarePairedDelimiter{\abs}{\lvert}{\rvert} 
\DeclarePairedDelimiter{\inner}{\langle}{\rangle}
\newcommand{\eref}[1]{\eqref{#1}}
\newcommand{\sref}[1]{section~\ref{#1}}
\newcommand{\fref}[1]{figure~\ref{#1}}

\newcommand{\Fref}[1]{Figure~\ref{#1}}
\newcommand{\absm}[1]{\abs{#1}}
\newcommand{\normm}[1]{\norm{#1}}
\newcommand{\innerm}[1]{\inner{#1}}
\newcommand{\gamD}{\Gamma^\textup{D}}

\newcommand{\gamN}{\Gamma^\textup{N}}
\newcommand{\para}[1]{\left(#1\right)}

\newcommand{\rum}[1]{\mathbb{#1}}

\newcommand{\arum}{\ensuremath{\mathcal{A}}}
\newcommand{\srum}{\ensuremath{\mathcal{S}}}

\newcommand{\ssc}[1]{\ensuremath{\mathcal{S}_{#1}}}

\newcommand{\dsr}{\delta\sigma}
\newcommand{\po}{{\partial\Omega}}
\newcommand{\fob}{\mathcal{F}}

\newcommand{\hmhalf}{H_\diamond^{-1/2}(\partial\Omega)}
\newcommand{\ds}{\eta}
\newcommand{\imbed}{\hookrightarrow}
\newcommand{\proja}{\mathcal{P}_{\arum_0}}

\DeclareMathOperator*{\supp}{supp}
\DeclareMathOperator*{\argmin}{argmin}

\DeclareMathOperator{\sign}{sgn}

\begin{document}

%\jvol{00} \jnum{00} \jyear{2014} \jmonth{March}

%\articletype{GUIDE}

\title{Sparsity prior for electrical impedance tomography with partial data}

\author{Henrik Garde and Kim Knudsen\\\vspace{6pt}  {\em{Department of Applied
  Mathematics and Computer Science, Technical University of Denmark,
 2800 Kgs. Lyngby, Denmark}}
\\\received{December 2014} }

\maketitle

\begin{abstract}
  This paper focuses on prior information for improved sparsity
  reconstruction in electrical impedance tomography with partial data,
  i.e.\ data measured only on subsets of the boundary. Sparsity is
  enforced using an $\ell_1$ norm of the basis coefficients as the
  penalty term in a Tikhonov functional, and prior information is
  incorporated by applying a spatially distributed regularization
  parameter. The resulting optimization problem allows great
  flexibility with respect to the choice of measurement boundaries and
  incorporation of prior knowledge. The problem is solved using a
  generalized conditional gradient method applying soft
  thresholding. Numerical examples show that the addition of prior
  information in the proposed algorithm gives vastly improved
  reconstructions even for the partial data problem. The method is in
  addition compared to a total variation approach.
\end{abstract}

\begin{keywords}
Electrical impedance tomography; inverse boundary value problem; ill-posed problem; partial data; sparsity
\end{keywords}

\begin{classcode}65N20; 65N21\end{classcode}

\section{Introduction}

The inverse problem in electrical impedance tomography (EIT) consists
of reconstructing an electrical conductivity distribution in the
interior of an object from electro-static boundary measurements on the
surface of the object. EIT is an emerging technology with applications
in medical imaging \cite{Holder2005}, geophysics \cite{Abubakar2009} and
industrial tomography \cite{York2001}. The underlying mathematical problem is
known as the Calder\'on problem in recognition of Calder\'on's seminal
paper \cite{Calderon1980}.

Consider a bounded domain $\Omega\subset \mathbb{R}^n,\; n\geq 2, $ with smooth boundary $\po.$ In
order to consider partial boundary measurements we introduce the
subsets $\gamN,\gamD \subseteq \po$ for the Neumann and Dirichlet data
respectively. Let $\sigma \in L^\infty(\Omega)$ with $0< c\leq \sigma$ a.e.
denote the conductivity distribution in $\Omega$. Applying a boundary
current flux $g$ (Neumann condition) through $\gamN\subseteq \po$
gives rise to the interior electric potential $u$ characterized as the
solution to
\begin{equation}
  \nabla \cdot(\sigma\nabla u) = 0 \text{ in } \Omega,\quad	\sigma \frac{\partial u}{\partial \nu} = g \text{ on } \partial\Omega, \quad \int_{\gamD} u|_{\po}\,ds = 0, \label{pde}
\end{equation}
where $\nu$ is an outward unit normal to $\po.$ The latter condition
in \eref{pde} is a grounding of the total electric potential along the
subset $\gamD\subseteq \po.$ To be precise we define the spaces 
\begin{align*}
  L_\diamond^2(\po) &\equiv \{g \in L^2(\partial\Omega) \mid \int_{\po}
  g \,ds = 0\}, \\ 
  \hmhalf &\equiv \{g \in H^{-1/2}(\partial\Omega)
  \mid \inner{g,1} = 0\},
\end{align*}
consisting of boundary functions with mean zero, and the spaces
\begin{align*}
  H_{\gamD}^1(\Omega) &\equiv \{u\in H^1(\Omega)\mid u|_{\po} \in
  H_{\gamD}^{1/2}(\po)\,\}, \\ 
  H_{\gamD}^{1/2}(\po) &\equiv \{f\in H^{1/2}(\po)\mid \int_{\gamD} f\,ds=0\},
\end{align*}
consisting of functions with mean zero on $\gamD$ designed to
encompass the partial boundary data.  Using standard elliptic theory
it follows that \eref{pde} has a unique solution $u \in
H_{\gamD}^1(\Omega)$ for any $g \in \hmhalf$. This defines the
Neumann-to-Dirichlet map (ND-map) $\Lambda_\sigma$ as an operator from
$H_\diamond^{-1/2}(\po)$ into $H_{\gamD}^{1/2}(\po)$ by $g \mapsto
u|_{\po}$, and the partial ND-map as $g \mapsto (\Lambda_\sigma
g)|_{\gamD}$.

The data for the classical Calder\'on problem is the full operator
$\Lambda_{\sigma}$ with $\gamD = \gamN = \partial\Omega.$
The problem is well-studied and there are numerous publications addressing different
aspects of its solution; we mention only a few: the uniqueness and
reconstruction problem was solved in \cite{SylvesterUhlmann1987,
  Nachman1988, Novikov1988, Nachman1996, AstalaPaivarinta2006,
  HabermanTataru2013} using the so called complex geometrical optics (CGO)
solutions; for a recent survey see \cite{Uhlmann2009}. Stability
estimates of log type were obtained in \cite{Alessandrini1988,
  Alessandrini1990} and shown to be optimal in
\cite{Mandache2001}. Thus any computational algorithm must rely on
regularization.  Such computational regularization algorithms
following the CGO approach were designed, implemented and analysed in
\cite{SiltanenMuellerIsaacson2001,KnudsenLassasMuellerSiltanen2009,
  BikowskiKnudsenMueller2011,DelbaryHansenKnudsen2012,DelbaryKnudsen2014}.

Recently the partial data Calder\'on problem has been studied
intensively. In 3D uniqueness has been proved under certain conditions
on $\gamD$ and $\gamN$
\cite{BukhgeimUhlmann2002,KenigSjostrandUhlmann2007,Knudsen2006,Zhang2012,Isakov2007}, and
in 2D the general problem with localized data i.e. $\gamD = \gamN =
\Gamma$ for some, possibly small, subset $\Gamma\subseteq\po$ has been
shown to posses uniqueness \cite{ImanuvilovUhlmannYamamoto2010}. Also
stability estimates of log-log type have been obtained for the partial
problem \cite{HeckWang2006}; this suggests that the partial data
problem is even more ill-posed and hence requires more regularization
than the full data problem. Recently a computational algorithm for the
partial data problem in 2D was suggested and investigated in
\cite{HamiltonSiltanen}.

A general approach to linear inverse problems with sparsity
regularization was given in \cite{daubechies2004}, and in
\cite{bredies2009,bonesky2007} the method
was adapted to non-linear problems using a so-called generalized conditional
gradient method. In \cite{gehre2012,jin2011,jin2012} the method was applied to the reconstruction problem in EIT with full boundary
data. For other approaches to EIT using optimization methods we refer
to \cite{Borcea2002}.

In this paper we will focus on the partial data problem for which we
develop a reconstruction algorithm based on a least squares
formulation with sparsity regularization. The results are twofold:
first we extend the full data algorithm of \cite{jin2012} to the case
of partial data, second we show how prior information about the spatial
location of the perturbation in the conductivity can be used in the
design of a spatially varying regularization parameter. We will restrict the treatment to 2D, however everything extends to 3D with some minor assumptions on the regularity of the Neumann data \cite{GardeKnudsen2014}.

The data considered here consist of a finite
number of Cauchy data taken on the subsets $\gamD$ and $\gamN,$ i.e.
\begin{equation}
	\{(f_k,g_k) \mid g_k \in \hmhalf, \; \supp(g_k)\subseteq \gamN, f_k =  \Lambda_{\sigma} g_k|_{\gamD} \}_{k=1}^K,\; K \in \mathbb{N}. \label{data}
\end{equation}
We assume that the true conductivity is given as $\sigma = \sigma_0 + \delta\sigma$, where $\sigma_0$ is a known background conductivity. Define the closed and convex subset 
\begin{equation}
	\arum_0 \equiv \{\delta\gamma \in H_0^1(\Omega)\mid c\leq \sigma_0+\delta\gamma\leq c^{-1} \text{  a.e. in } \Omega\}  \label{a0ref}
\end{equation}
for some $c \in (0,1)$, and $\sigma_0\in H^1(\Omega)$ where $c\leq \sigma_0\leq c^{-1}$. Similarly define 
\[
	\arum\equiv \arum_0 + \sigma_0 = \{\gamma\in H^1(\Omega)\mid c\leq \gamma\leq c^{-1} \text{ a.e. in } \Omega, \gamma|_{\po} = \sigma_0|_{\po}\}.
\]
The inverse problem is then to approximate $\delta\sigma\in\arum_0$ given the data \eref{data}.

Let $\{\psi_j\}$ denote a chosen orthonormal basis for $H_0^1(\Omega).$ For sparsity regularization we approximate $\delta\sigma$ by $\argmin_{\delta\gamma\in\arum_0}\Psi(\delta\gamma)$ using the following Tikhonov functional \cite{jin2012}
\begin{equation}
	\Psi(\delta\gamma) \equiv \sum_{k=1}^K R_k(\delta\gamma) + P(\delta\gamma), \enskip\delta\gamma\in\arum_0, \label{psieq}
\end{equation} 
with
\[
	R_k(\delta\gamma) \equiv \frac{1}{2}\normm{\Lambda_{\sigma_0+\delta\gamma}g_k - f_k}_{L^2(\gamD)}^2, \quad P(\delta\gamma) \equiv \sum_{j=1}^\infty \alpha_j\absm{c_j},
\]
for $c_j \equiv \inner{\delta\gamma,\psi_j}$. The regularization parameter $\alpha_j>0$ for the
sparsity-promoting $\ell_1$ penalty term $P$ is distributed such that
each basis coefficient can be regularized differently; we will return to
this in \sref{sec:prior}. It should be noted how easy and natural the
use of partial data is introduced in this way, simply by only
minimizing the discrepancy on $\gamD$ on which the Dirichlet data is known and ignoring the rest of the boundary.

This paper is organised as follows: in \sref{sec:SparseReconstruction}
we derive the Fr\'echet derivative of $R_k$ and reformulate the
optimization problem using the generalized conditional gradient method as a
sequence of linearized optimization problems. In \sref{sec:prior} we
explain the idea of the spatially dependent regularization parameter
designed for the use of prior information. Then in
\sref{sec:NumericalResults} we show the feasibility of the algorithm
by several numerical examples, and finally we conclude in
\sref{sec:conclusions}.

\section{Sparse Reconstruction}\label{sec:SparseReconstruction}

In this section the sparse reconstruction of $\dsr$ based on the optimization problem \eqref{psieq}, is investigated for a bounded domain $\Omega\subset\rum{R}^2$ with smooth boundary $\partial\Omega$. The penalty term emphasizes that $\dsr$ should only be expanded by few basis functions in a given orthonormal basis. Using a distributed regularization parameter, it is
possible to further apply prior information about which basis functions that should be included in the expansion of $\dsr$. The partial data problem comes into play in the discrepancy term, in which we only fit the data on part of the boundary. Ultimately, this leads to the algorithm given in Algorithm \ref{alg1} at the end of this section. 

Denote by $F_g(\sigma)$ the unique solution to \eref{pde} and let $\fob_g(\sigma)$ be its trace (note that $\Lambda_\sigma g = \fob_g(\sigma)$). Let $\gamma\in\arum$, $g\in L^p(\po)\cap\hmhalf$ for $p>1$, then following the proofs of Theorem 2.2 and Corollary 2.1 in \cite{jin2011} whilst applying the partial boundary $\gamD$ we have
\begin{equation}
	\lim_{\substack{\normm{\ds}_{H^1(\Omega)}\to 0 \\ \gamma+\ds\in\arum}}\frac{\normm{\fob_g(\gamma+\ds)-\fob_g(\gamma) - (\fob_g)'_\gamma\ds}_{H^{1/2}_{\gamD}( \partial\Omega)}}{\normm{\ds}_{H^1(\Omega)}} = 0. \label{fpplimit}
\end{equation}
Here $(\fob_g)'_\gamma$ is the linear map, that maps $\eta$ to $w|_{\po}$, where $w$ is the unique solution to
\begin{equation}
	-\nabla\cdot(\gamma\nabla w) = \nabla\cdot(\ds\nabla F_g(\gamma)) \text{ in } \Omega, \quad \sigma\frac{\partial w}{\partial\nu} = 0 \text{ on } \partial\Omega, \quad \int_{\gamD}w|_{\partial\Omega}\,ds = 0. \label{fprimepde}
\end{equation}
It is noted that $(\fob_g)'_\gamma$ resembles a Fr\'echet derivative of $\fob_g$ evaluated at $\gamma$ due to \eref{fpplimit}, however $\arum$ is not a linear vector space, thus the requirement $\gamma,\gamma+\ds\in\arum$.

The first step in minimizing $\Psi$ using a gradient descent type iterative algorithm is to determine a derivative to the discrepancy terms $R_k$.
\begin{lemma} 
	Let $\gamma=\sigma_0+\delta\gamma$ for $\delta\gamma\in\arum_0$, and $\chi_{\gamD}$ be a characteristic function on $\gamD$. Then 
	\begin{equation}
		G_k \equiv -\nabla F_{g_k}(\gamma)\cdot\nabla F_{\chi_{\gamD}(\Lambda_\gamma g_k-f_k)}(\gamma)\in L^r(\Omega)\subset H^{-1}(\Omega) \label{nablaJ}
	\end{equation}
	for some $r>1$, and the Fr\'echet derivative $(R_k)'_{\delta\gamma}$ of $R_k$ on $H_0^1(\Omega)$ evaluated at $\delta\gamma$ is given by
	\begin{equation}
		(R_k)'_{\delta\gamma}\eta = \int_{\Omega} G_k\ds \,dx, \enskip \delta\gamma+\ds\in \arum_0. \label{Jprimeeq}
	\end{equation}
\end{lemma}
\begin{proof}
	For the proof the index $k$ is suppressed. First it is proved that $G\in L^r(\Omega)$ for some $r>1$, which is shown by estimates on $F_g(\gamma)$ and $F_h(\gamma)$ where $h \equiv \chi_{\gamD}(\Lambda_\gamma g-f)$. Note that $\Lambda_\gamma g \in H_{\gamD}^{1/2}(\po)$ and $f\in L^2_\diamond(\gamD)$, i.e. $h \in L^2_\diamond(\po) \subset L^2(\po)\cap H_\diamond^{-1/2}(\po)$. Now using \cite[Theorem 3.1]{jin2011}, there exists $Q>2$ such that
	\begin{equation}
		\normm{F_h(\gamma)}_{W^{1,q}(\Omega)} \leq C\normm{h}_{L^2(\po)}, \label{Fweq}
	\end{equation}
	where $q\in(2,Q)\cap [2,4]$. Since $F_g(\gamma)\in H_{\gamD}^1(\Omega)$ then $\absm{\nabla F_g(\gamma)} \in L^2(\Omega)$. It has already been established in \eref{Fweq} that $F_h(\gamma) \in W^{1,q}(\Omega)$ for $q\in(2,\min\{Q,4\})$, so $\absm{\nabla F_h(\gamma)}\in L^q(\Omega)$. By H{\"o}lder's generalized inequality
	\[
		G = -\nabla F_g(\gamma)\cdot \nabla F_h(\gamma) \in L^r(\Omega), \enskip \tfrac{1}{r} = \tfrac{1}{2} + \tfrac{1}{q},
	\]
	and as $q > 2$ then $r > 1$. Let $r'$ be the conjugate exponent to $r$, then $r'\in[1,\infty)$, i.e. the Sobolev imbedding theorem \cite{sobolev} implies that $H^1(\Omega)\imbed L^{r'}(\Omega)$ as $\Omega\subset\rum{R}^2$. Thus $G \in (L^{r'}(\Omega))'\subset (H^1(\Omega))' \subset (H_0^1(\Omega))' = H^{-1}(\Omega)$.
	
	Now it will be shown that $R'_{\delta\gamma}$ can be identified with $G$. $R'_{\delta\gamma}\eta$ is by the chain rule (utilizing that $\Lambda_\gamma g = \fob_g(\gamma)$) given as
	\begin{equation}
		R'_{\delta\gamma}\eta = \int_\po \chi_{\gamD}(\Lambda_\gamma g-f)(\fob_{g})_\gamma'\ds \,ds, \label{expectedprime}
	\end{equation}
	where $\chi_{\gamD}$ is enforcing that the integral is over $\gamD$. The weak formulations of \eref{pde}, with Neumann data $\chi_{\gamD}(\Lambda_\gamma g-f)$, and \eref{fprimepde} are
	\begin{align}
		\int_{\Omega} \gamma \nabla F_{\chi_{\gamD}(\Lambda_\gamma g-f)}(\gamma)\cdot \nabla v\,dx &= \int_{\po} \chi_{\gamD}(\Lambda_\gamma g-f)v|_{\po}\,ds, \enskip \forall v\in H^1(\Omega), \label{weak1} \\
		\int_{\Omega} \gamma \nabla w\cdot \nabla v\,dx &= -\int_{\Omega} \ds\nabla F_{g}(\gamma)\cdot \nabla v\,dx, \enskip \forall v\in H^1(\Omega). \label{weak2}
	\end{align}
	Now by letting $v \equiv w$ in \eref{weak1} and $v \equiv F_{\chi_{\gamD}(\Lambda_\gamma g-f)}(\gamma)$ in \eref{weak2}, we obtain using the definition $w|_{\po} = (\fob_g)'_\gamma\eta$ that
	\begin{align*}
		R'_{\delta\gamma}\ds &= \int_\po \chi_{\gamD}(\Lambda_\gamma g-f)(\fob_{g})_\gamma'\ds \,ds = \int_{\Omega} \gamma \nabla F_{\chi_{\gamD}(\Lambda_\gamma g-f)}(\gamma)\cdot \nabla w\,dx \\
		&=  -\int_{\Omega} \ds\nabla F_{g}(\gamma)\cdot \nabla F_{\chi_{\gamD}(\Lambda_\gamma g-f)}(\gamma)\,dx = \int_{\Omega}G\ds \,dx.
	\end{align*}
\end{proof}
\begin{remark}
	It should be noted that $(R_k)'_{\delta\gamma}$ is related to the Fr\'echet derivative $\Lambda_\gamma'$ of $\gamma\mapsto \Lambda_\gamma$ evaluated at $\gamma$, by $(R_k)'_{\delta\gamma}\eta = \int_{\gamD}(\Lambda_\gamma g_k-f_k)\Lambda_\gamma'[\eta]g_k ds$.
\end{remark}
Define 
\[
	R'_{\delta\gamma} \equiv \sum_{k=1}^K (R_k)'_{\delta\gamma} = -\sum_{k=1}^K\nabla F_{g_k}(\gamma)\cdot\nabla F_{\chi_{\gamD}(\Lambda_{\gamma} g_k-f_k)}(\gamma).
\]
For a gradient type descent method, we seek to find a direction $\eta$ for which the discrepancy decreases. As $R_{\delta\gamma}'\in H^{-1}(\Omega)$ it is known from Riesz' representation theorem that there exists a unique function in $H_0^1(\Omega)$, denoted by $\nabla_s R(\delta\gamma)$, such that 
\begin{equation}
	R_{\delta\gamma}'\eta = \inner{\nabla_s R(\delta\gamma),\eta}_{H^1(\Omega)},\enskip \eta\in H_0^1(\Omega). \label{sobograd}
\end{equation}
Now $\eta \equiv -\nabla_s R(\delta\gamma)$ points in the steepest descend direction among the viable directions. Furthermore, since $\nabla_s R(\delta\gamma)|_{\po} = 0$ the boundary condition $\delta\sigma|_{\po} = 0$ will automatically be fulfilled for the approximation. In \cite{sobolevgradient} $\nabla_s R(\delta\gamma)$ is called a Sobolev-gradient, and it is the unique solution to 
\[
	(-\Delta+1)v=R_{\delta\gamma}' \text{ in } \Omega, \quad v = 0\text{ on } \po,
\]
for which \eref{sobograd} is the weak formulation. In each iteration step we need to determine a step size $s_i$ for an algorithm resembling a steepest descent $\delta\gamma_{i+1} = \delta\gamma_i - s_i\nabla_sR(\delta\gamma_i)$. Here a Barzilai-Borwein step size rule \cite{BBrules,sparsa,jin2012} will be applied, for which we determine $s_i$ such that $\frac{1}{s_i}(\delta\gamma_i-\delta\gamma_{i-1}) = \frac{1}{s_i}(\gamma_i-\gamma_{i-1})\simeq \nabla_s R(\delta\gamma_i)-\nabla_s R(\delta\gamma_{i-1})$ in the least-squares sense
\begin{equation}
	s_i \equiv \argmin_{s} \normm{s^{-1}(\delta\gamma_i-\delta\gamma_{i-1})-(\nabla_s R(\delta\gamma_i)-\nabla_s R(\delta\gamma_{i-1}))}_{H^1(\Omega)}^2. \label{bbrule}
\end{equation}
Assuming that $\innerm{\delta\gamma_i-\delta\gamma_{i-1},\nabla_s R(\delta\gamma_i)-\nabla R(\delta\gamma_{i-1})}_{H^1(\Omega)}\neq 0$ yields
\begin{equation}
	s_i = \frac{\normm{\delta\gamma_i-\delta\gamma_{i-1}}_{H^1(\Omega)}^2}{\innerm{\delta\gamma_i-\delta\gamma_{i-1},\nabla_s R(\delta\gamma_i)-\nabla_sR(\delta\gamma_{i-1})}_{H^1(\Omega)}}. \label{bbstepsize}
\end{equation}
A maximum step size $s_{\max}$ is enforced to avoid the situations where $\innerm{\delta\gamma_i-\delta\gamma_{i-1},\nabla_s R(\delta\gamma_i)-\nabla R(\delta\gamma_{i-1})}_{H^1(\Omega)}\simeq 0$.	

With inspiration from \cite{sparsa}, $s_i$ will be initialized by \eref{bbstepsize}, after which it is thresholded to lie in $[s_{\min},s_{\max}]$, for positive constants $s_{\min}$ and $s_{\max}$. It is noted in \cite{sparsa} that Barzilai-Borwein type step rules lead to faster convergence if we do not restrict $\Psi$ to decrease in every iteration. Allowing an occasional increase in $\Psi$ can be used to avoid places where the method has to take many small steps to ensure the decrease of $\Psi$. Therefore, one makes sure that the following so called weak monotonicity is satisfied, which compares $\Psi(\delta\gamma_{i+1})$ with the most recent $M$ steps.	Let $\tau\in(0,1)$ and $M\in\rum{N}$, then $s_i$ is said to satisfy the weak monotonicity with respect to $M$ and $\tau$ if the following is satisfied \cite{sparsa}
\begin{equation}
	\Psi(\delta\gamma_{i+1}) \leq \max_{i-M+1\leq j \leq i}\Psi(\delta\gamma_j) - \frac{\tau}{2s_i}\normm{\delta\gamma_{i+1}-\delta\gamma_i}_{H^1(\Omega)}^2.\label{weakmono}
\end{equation}
If \eref{weakmono} is not satisfied, the step size $s_i$ is reduced until this is the case. To solve the non-linear minimization problem for \eqref{psieq} we iteratively solve the following linearized problem
\begin{align}
	\zeta_{i+1} &\equiv \argmin_{\delta\gamma\in H_0^1(\Omega)}\left[\frac{1}{2}\normm{\delta\gamma - (\delta\gamma_i-s_i\nabla_sR(\delta\gamma_i))}_{H^1(\Omega)}^2 + s_i\sum_{j=1}^\infty \alpha_j\absm{c_j}\right], \label{upsiloneq} \\
	\delta\gamma_{i+1} &\equiv \proja(\zeta_{i+1}). \notag
\end{align}
Here $\{\psi_j\}$ is an orthonormal basis for $H_0^1(\Omega)$ in the $H^1$-metric, and $\proja$ is a projection of $H_0^1(\Omega)$ onto $\arum_0$ to ensure that \eqref{pde} is solvable (note that $H_0^1(\Omega)$ does not imbed into $L^\infty(\Omega)$, i.e. $\zeta_{i+1}$ may be unbounded). By use of the map $\srum_\beta:\rum{R}\to\rum{R}$ defined below, known as the soft shrinkage/thresholding map with threshold $\beta > 0$,
\begin{equation}
	\ssc{\beta}(x) \equiv \sign(x)\max\{\absm{x}-\beta,0\},\enskip x\in\rum{R}, \label{softoperator}
\end{equation}
the solution to \eref{upsiloneq} is easy to find directly (see also \cite[Section 1.5]{daubechies2004}):
\begin{equation}
	\zeta_{i+1} = \sum_{j=1}^\infty \ssc{s_i\alpha_j}(d_j)\psi_j, \label{softstep}
\end{equation}
where $d_j\equiv\innerm{\delta\gamma_i-s_i\nabla_sR(\delta\gamma_i),\psi_j}_{H^1(\Omega)}$ are the basis coefficients for $\delta\gamma_i-s_i\nabla_sR(\delta\gamma_i)$. 

The projection $\proja : H_0^1(\Omega)\to \arum_0$ is defined as
\begin{equation*}
	\proja(v) \equiv T_c(\sigma_0 + v) - \sigma_0, \enskip v\in H_0^1(\Omega),
\end{equation*}
where $T_c$ is the following truncation that depends on the constant $c\in(0,1)$ in \eref{a0ref}
\begin{equation*}
	T_c(v) \equiv \begin{cases}
		c & \text{where } v < c \text{ a.e.}, \\
		c^{-1} & \text{where } v > c^{-1} \text{ a.e.}, \\
		v & \text{else.} 
	\end{cases}
\end{equation*}
Since $\sigma_0\in H^1(\Omega)$ and $c\leq \sigma_0\leq c^{-1}$, it follows directly from \cite[Lemma 1.2]{Stampacchia_1965} that $T_c$ and $\proja$ are well-defined, and it is easy to see that $\proja$ is a projection. It should also be noted that $0\in \arum_0$ since $c\leq \sigma_0\leq c^{-1}$, thus we may choose $\delta\gamma_0 \equiv 0$ as the initial guess in the algorithm.

The algorithm is summarized in Algorithm \ref{alg1}. In this paper the stopping criteria is when the step size $s_i$ gets below a threshold $s_\text{stop}$.
%
%\vspace{-15pt}
%
\begin{remark}
	Note that $\sum_j \innerm{\delta\gamma_i-s_i\nabla_sR(\delta\gamma_i) , \psi_j}_{H^1(\Omega)} \psi_j$ corresponds to only having the discrepancy term in \eref{upsiloneq}, while the penalty term corresponds to changing these coefficients with the soft thresholding.
\end{remark}
\begin{remark}
	The non-linearity of $\gamma\mapsto\Lambda_\gamma$ leads to a non-convex discrepancy term, i.e. $\Psi$ is non-convex. So the best we can hope is to find a local minimum.
\end{remark}
\begin{algorithm} \caption{Sparse Reconstruction for Partial Data EIT} \label{alg1}
\begin{algorithmic}
	\State Set $\delta\gamma_0 := 0$.
	\While{stopping criteria not reached}
		\State Set $\gamma_i := \sigma_0 + \delta\gamma_i$.
		\State Compute $\Psi(\delta\gamma_i)$.
		\State Compute $R_{\delta\gamma_i}' := -\sum_{k=1}^K\nabla F_{g_k}(\gamma_i)\cdot\nabla F_{\chi_{\gamD}(\Lambda_{\gamma_i} g_k-f_k)}(\gamma_i)$.
		\State Compute $\nabla_s R(\delta\gamma_i)\in H_0^1(\Omega)$ such that $R_{\delta\gamma_i}'\ds = \innerm{\nabla_s R(\delta\gamma_i),\ds}_{H^1(\Omega)}$.
		\State Compute step length $s_i$ by \eref{bbstepsize}, and decrease it till \eref{weakmono} is satisfied.
		\State Compute the basis coefficients $\{d_j\}_{j=1}^\infty$ for $\delta\gamma_i-s_i\nabla_sR(\delta\gamma_i)$.
		\State Update $\delta\gamma_{i+1} := \proja\para{\sum_{j=1}^\infty \ssc{s_i\alpha_j}(d_j)\psi_j}$.	
	\EndWhile
	\State Return final iterate of $\delta\gamma$.
\end{algorithmic}
\end{algorithm}
\begin{remark}
	The main computational cost lies in computing $R_{\delta\gamma_i}'$, which involves solving $2K$ well-posed PDE's (note that $F_{g_k}(\gamma_i)$ can be reused from the evaluation of $\Psi$). It should be noted that each of the $2K$ problems consists of solving the same problem, but with different boundary conditions, which leads to only having to assemble and factorize the FEM matrix once per iteration. 
\end{remark}

\section{Prior Information} \label{sec:prior}

Prior information is typically introduced in the penalty term $P$ for Tikhonov-like functionals, and here the regularization parameter determines how much this prior information is enforced. In the case of sparsity regularization this implies knowledge of how sparse we expect the solution is in general. Instead of applying the same prior information for each basis function, a distributed parameter is applied. Let
\[
	\alpha_j \equiv \alpha\mu_j,
\] 
where $\alpha$ is a usual regularization parameter, corresponding to the case where no prior information is considered about specific basis functions. The $\mu_j\in(0,1]$ will be used to weigh the penalty depending on whether a specific basis function should be included in the expansion of $\dsr$. The $\mu_j$ are chosen as
\[
\mu_j = \begin{cases}
	1, \quad & \text{no prior on $c_j$}, \\ \sim 0, & \text{prior that $c_j\neq 0$},	
\end{cases}
\]
i.e. if we know that a coefficient in the expansion of $\dsr$ should be non-zero, we can choose to penalize that coefficient less. 

\subsection{Applying the FEM basis}

In order to improve the sparsity solution for finding small
inclusions, it seems appropriate to include prior information about
the support of the inclusions. There are different methods available
for obtaining such information assuming piecewise constant
conductivity \cite{kirsch2007,HarrachUllrich2013} or real analytic
conductivity \cite{HarrachSeo2010}. An example of the reconstruction of $\supp\delta\sigma$ is shown in \fref{fig:monomethod}, where it is observed that numerically it is possible to reconstruct a reasonable convex approximation to the support. Thus, it is possible to acquire estimates of $\supp \delta\sigma$ \emph{for free}, in the sense that it is gained directly from the data without further assumptions.

\begin{figure}[h]%
\centering
\begin{subfigure}[b]{.33\linewidth}
	\includegraphics[width = \linewidth, trim = 4cm 4cm 1cm 4cm, clip=true]{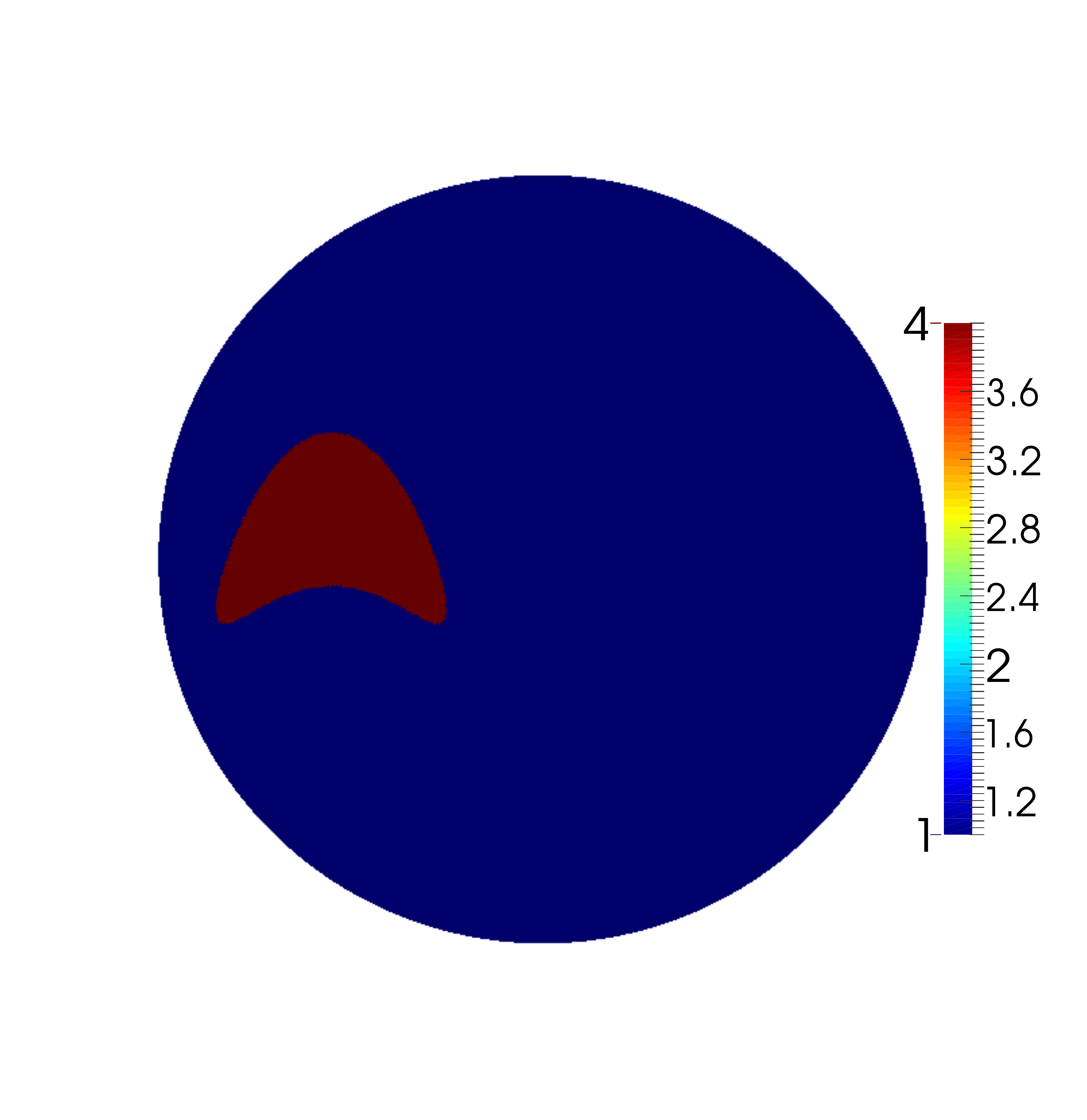}
	\caption{{}} \label{fig:kitephantom_a}
\end{subfigure}%
\hspace{2cm}
\begin{subfigure}[b]{.33\linewidth}
	\includegraphics[width = \linewidth, trim = 4cm 4cm 1cm 4cm, clip=true]{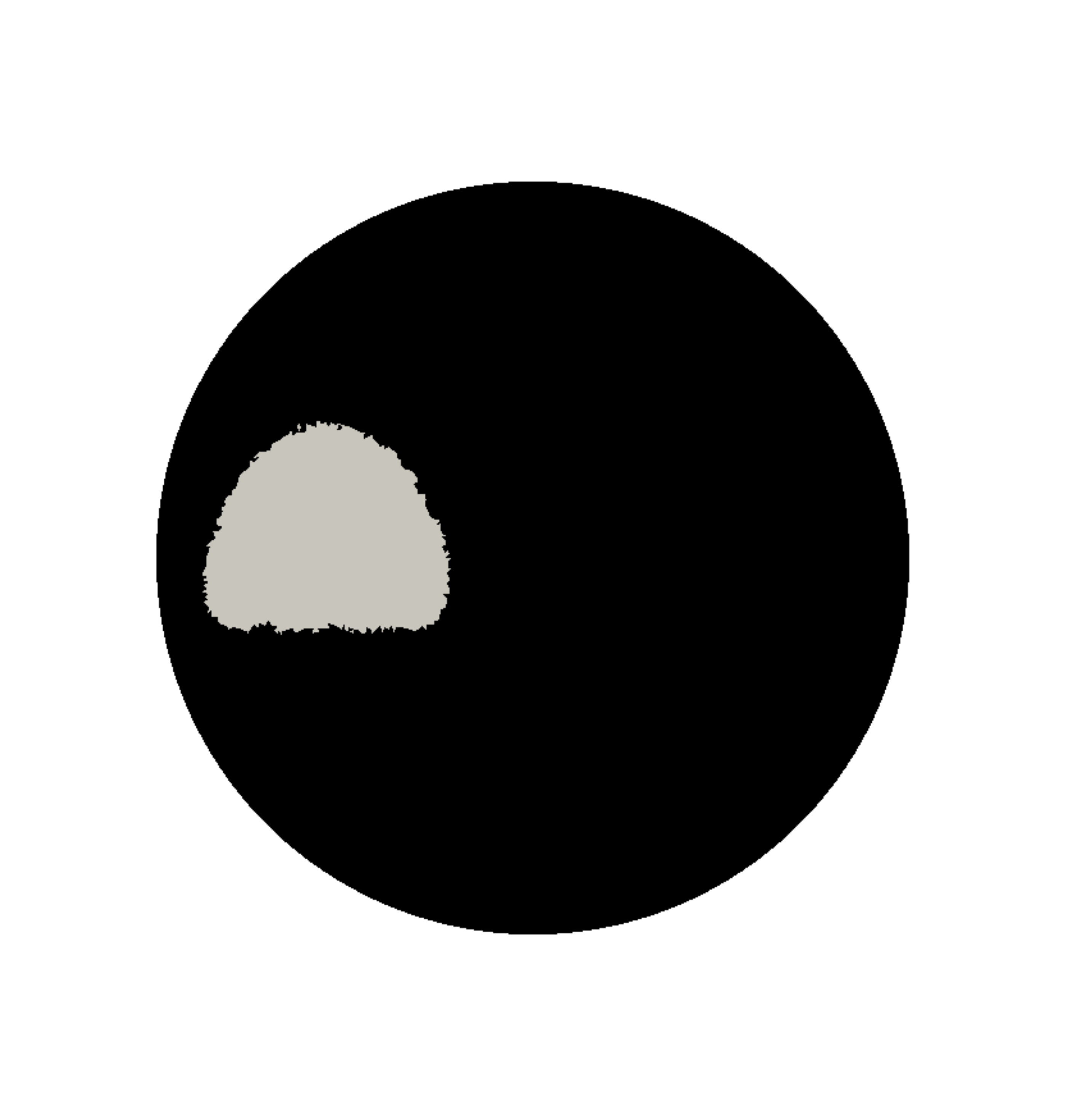}
	\caption{{}} \label{fig:mono}
\end{subfigure}%
\caption{\textbf{(a):} Phantom with kite-shaped piecewise constant inclusion $\delta\sigma$. \textbf{(b):} Reconstruction of $\supp \delta\sigma$ using monotonicity relations from the approach in \cite{HarrachUllrich2013} by use of simulated noiseless data.}
\label{fig:monomethod}
\end{figure}

Another approach is to consider
other reconstruction methods such as total variation regularization
that tends to give good approximations to the support, but has issues
with reconstructing the contrast if the amplitude of $\dsr$ is
large as seen in \sref{sec:comparison}. The idea is to be able to apply such information in the
sparsity algorithm in order to get good contrast reconstruction while
maintaining the correct support, even for the partial data problem. 

Suppose that as a basis we consider a finite element method (FEM)
basis $\{\psi_j\}_{j=1}^N$ for the subspace $V_h\subseteq H_0^1(\Omega)$ of piecewise affine
elements. This basis comprises basis functions that are piecewise
affine with degrees of freedom at the mesh nodes,
i.e.\ $\psi_j(x_k) = \delta_{j,k}$ at mesh node $x_k$ in the applied
mesh. Let $\dsr\in V_h$, then $\dsr(x) = \sum_j \dsr(x_j)\psi_j(x)$,
i.e.\ for each node there is a basis function for which the coefficient
contains local information about the expanded function; this is
convenient when applying prior information about the support of an
inclusion. Note that the FEM basis functions are not mutually orthogonal, since basis functions corresponding to neighbouring nodes are non-negative and have
overlapping support. However, for any non-neighbouring pair of nodes
the corresponding basis functions are orthogonal.

When applying the FEM basis for mesh nodes $\{x_j\}_{j=1}^N$, the corresponding functional is 
\[
	\Psi(\delta\gamma) = \frac{1}{2}\sum_{k=1}^K \normm{\Lambda_{\sigma_0+\delta\gamma}g_k - f_k}_{L^2(\gamD)}^2 + \sum_{j=1}^N \alpha_j\absm{\delta\gamma(x_j)}.
\]
It is evident that the penalty corresponds to determining inclusions with small support, and prior information on the sparsity corresponds to prior information on the support of $\dsr$. We cannot directly utilize \eref{softstep} due to the FEM basis not being an orthonormal basis for $H_0^1(\Omega)$, and instead we suggest the following iteration step:
\begin{align}
	\zeta_{i+1}(x_j) &= \ssc{s_i\alpha_j/\normm{\psi_j}_{L^1(\Omega)}}(\delta\gamma_i(x_j)-s_i\nabla_sR(\delta\gamma_i)(x_j)),\enskip j=1,2,\dots,N, \label{femupdate} \\
	\delta\gamma_{i+1} &= \proja(\zeta_{i+1}). \notag
\end{align}
Note that the regularization parameter will depend quite heavily on the discretization of the mesh, i.e. for the same domain a good regularization parameter $\alpha$ will be much larger on a coarse mesh than on a fine mesh. This is quite inconvenient, and instead we can weigh the regularization parameter according to the mesh cells, by having $\alpha_j \equiv \alpha\beta_j\mu_j$. This leads to a discretization of a weighted $L^1$-norm penalty term:
\[
	\alpha\int_{\Omega} f_\mu \absm{\dsr}\,dx \simeq \alpha\sum_j \beta_j\mu_j\absm{\dsr(x_j)}, 
\]
where $f_\mu : \Omega\to (0,1]$ is continuous and $f_\mu(x_j) = \mu_j$. For a triangulated mesh, the weights $\beta_j$ consists of the node area computed in 2D as 1/3 of the area of $\supp\psi_j$. This corresponds to splitting each cell's area evenly amongst the nodes, and it will not lead to instability on a regular mesh. This will make the choice of $\alpha$ almost independent of the mesh.
\begin{remark}
	The corresponding algorithm with the FEM basis is the same as Algorithm \ref{alg1}, except that the update is applied via \eref{femupdate}.
\end{remark}

\section{Numerical Examples}\label{sec:NumericalResults}

In this section we illustrate, through several examples, the numerical algorithm implemented using the finite element library FEniCS \cite{logg2012a}. First we consider the full data case $\gamD = \gamN = \partial\Omega$ without and with prior information, and then we do the same for the partial data case. Finally, a brief comparison is made with another sparsity promoting method based on total variation. 

For the following examples $\Omega$ is the unit disk in
$\rum{R}^2$. The regularization parameter $\alpha$ is chosen manually
by trial and error.  The other parameters are $\sigma_0\equiv 1$, $M =
5$, $\tau = 10^{-5}$, $s_{\min} = 1$, $s_{\max} = 1000$, and the
stopping criteria is when the step size is below $s_{\text{stop}} =
10^{-3}$. $K = 10$ and the applied Neumann data will be of the form
$g_n^\textup{c}(\theta) \equiv \cos(n\theta)$ and
$g_n^\textup{s}(\theta) \equiv \sin(n\theta)$ for $n = 1,\dots,5$ and
$\theta$ being the angular variable. For the partial data an interval
$\Gamma = \gamN = \gamD = \{\theta\in(\theta_1,\theta_2)\}$ is
considered, and $g_n^\textup{c}$ and $g_n^\textup{s}$ are scaled and
translated such that they have $n$ periods in the interval.

When applying prior information, the coefficients $\mu_j$ are chosen as $10^{-2}$ where the support of $\dsr$ is assumed, and $1$ elsewhere. It should be noted that in order to get fast transitions for sharp edges when prior information is applied, a local mesh refinement is used during the iterations to refine the mesh where $\absm{\nabla\dsr}$ is large.

For the simulated Dirichlet data, the forward problem is computed on a very fine mesh, and afterwards interpolated onto a different much coarser mesh in order to avoid inverse crimes. White Gaussian noise has been added to the Dirichlet data $\{f_k\}_{k=1}^K$ on the discrete nodes on the boundary of the mesh. The standard deviation of the noise is chosen as $\epsilon \max_{k}\max_{x_j\in\gamD}\absm{f_k(x_j)}$ as in \cite{jin2012}, where the noise level is fixed as $\epsilon = 10^{-2}$ (corresponding to 1\% noise) unless otherwise stated.

\begin{figure}[h]%
\centering
\begin{subfigure}[b]{.33\linewidth}
	\includegraphics[width = \linewidth, trim = 4cm 4cm 1cm 4cm, clip=true]{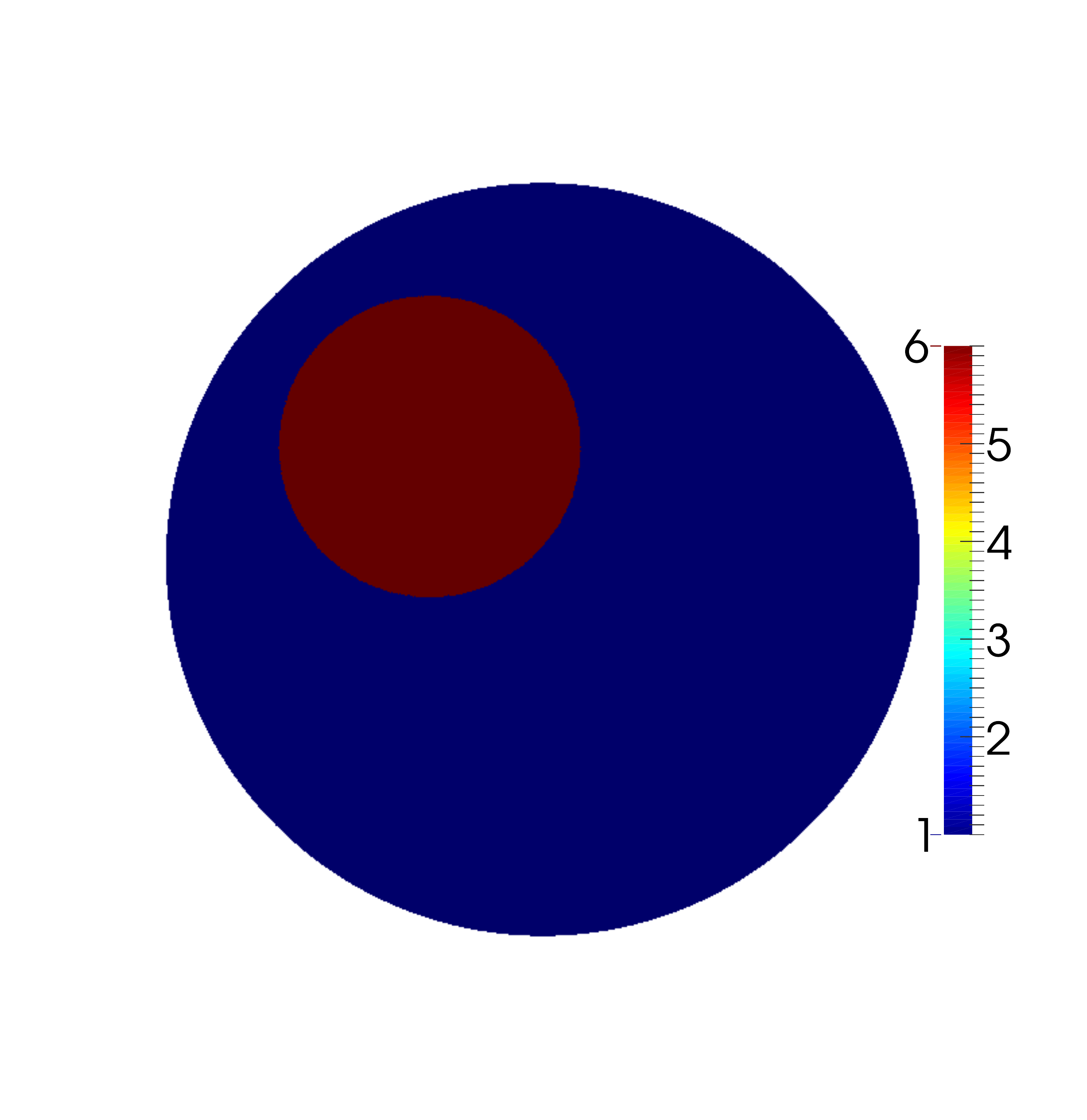}
	\caption{{}} \label{fig:circphantom}
\end{subfigure}%
\hfill
\begin{subfigure}[b]{.33\linewidth}
	\includegraphics[width = \linewidth, trim = 4cm 4cm 1cm 4cm, clip=true]{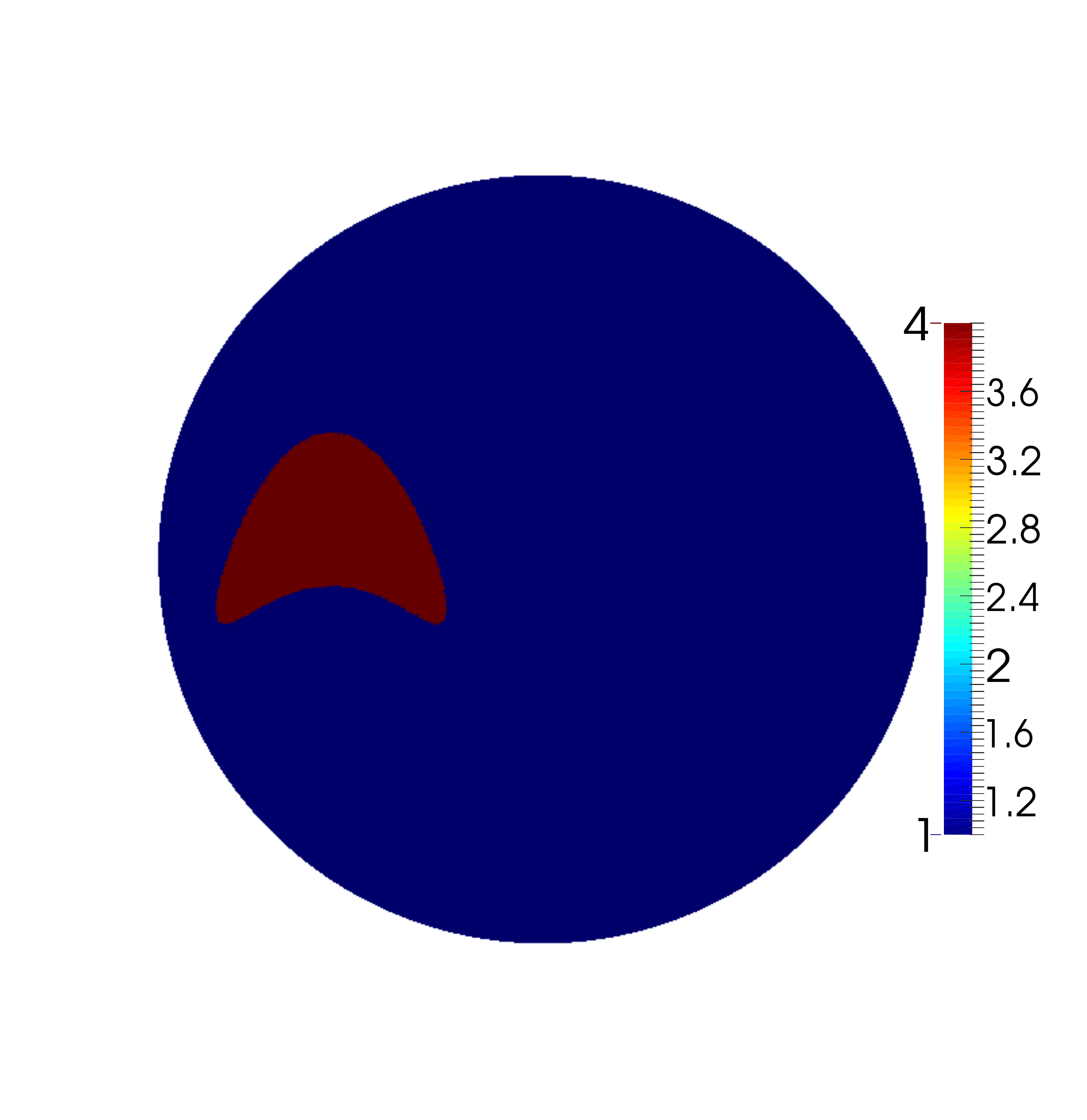}
	\caption{{}} \label{fig:kitephantom}
\end{subfigure}%
\hfill
\begin{subfigure}[b]{.33\linewidth}
	\includegraphics[width = \linewidth, trim = 4cm 4cm 1cm 4cm, clip=true]{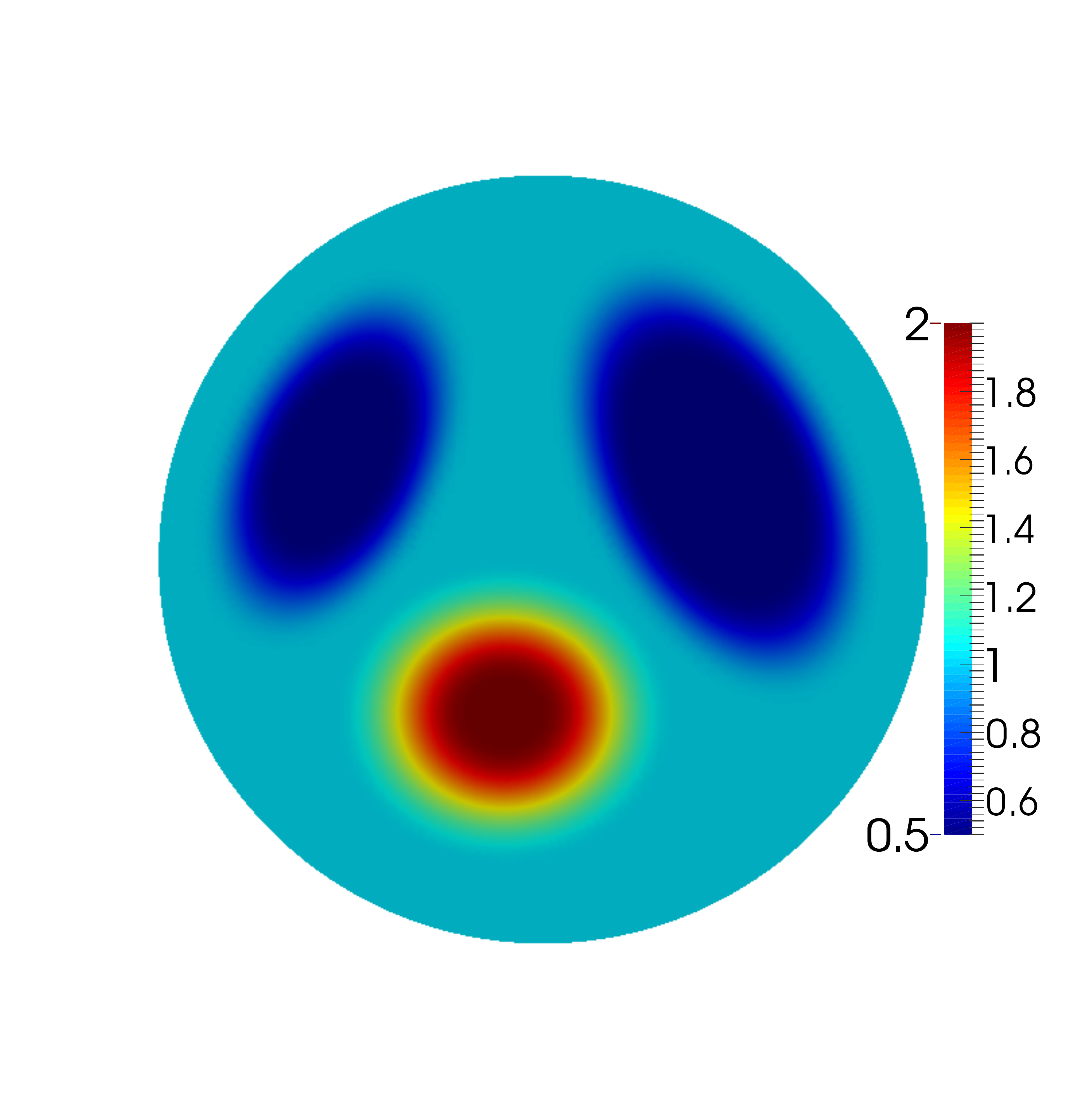}
	\caption{{}} \label{fig:hnlphantom}
\end{subfigure}% 
\caption{\textbf{(a):} Circular piecewise constant inclusion. \textbf{(b):} Kite-shaped piecewise constant inclusion. \textbf{(c):} Multiple $C^2$ inclusions.}
\label{fig:phantoms}
\end{figure}

\Fref{fig:phantoms} shows the numerical phantoms: where one is a simple circular inclusion, another is the non-convex kite-shaped phantom. Finally, we also shortly investigate the case of multiple smoother inclusions. 

\subsection{Full Boundary Data}

For $\gamD = \gamN = \po$ it is possible to get quite good reconstructions of both shape and contrast for the convex inclusions as seen in \fref{fig:recon1}, and for the case with multiple inclusions there is a reasonable separation of the inclusions.

\begin{figure}[h]%
\centering
\begin{subfigure}[h]{.33\linewidth}
	\includegraphics[width = \linewidth, trim = 4cm 4cm 1cm 4cm, clip=true]{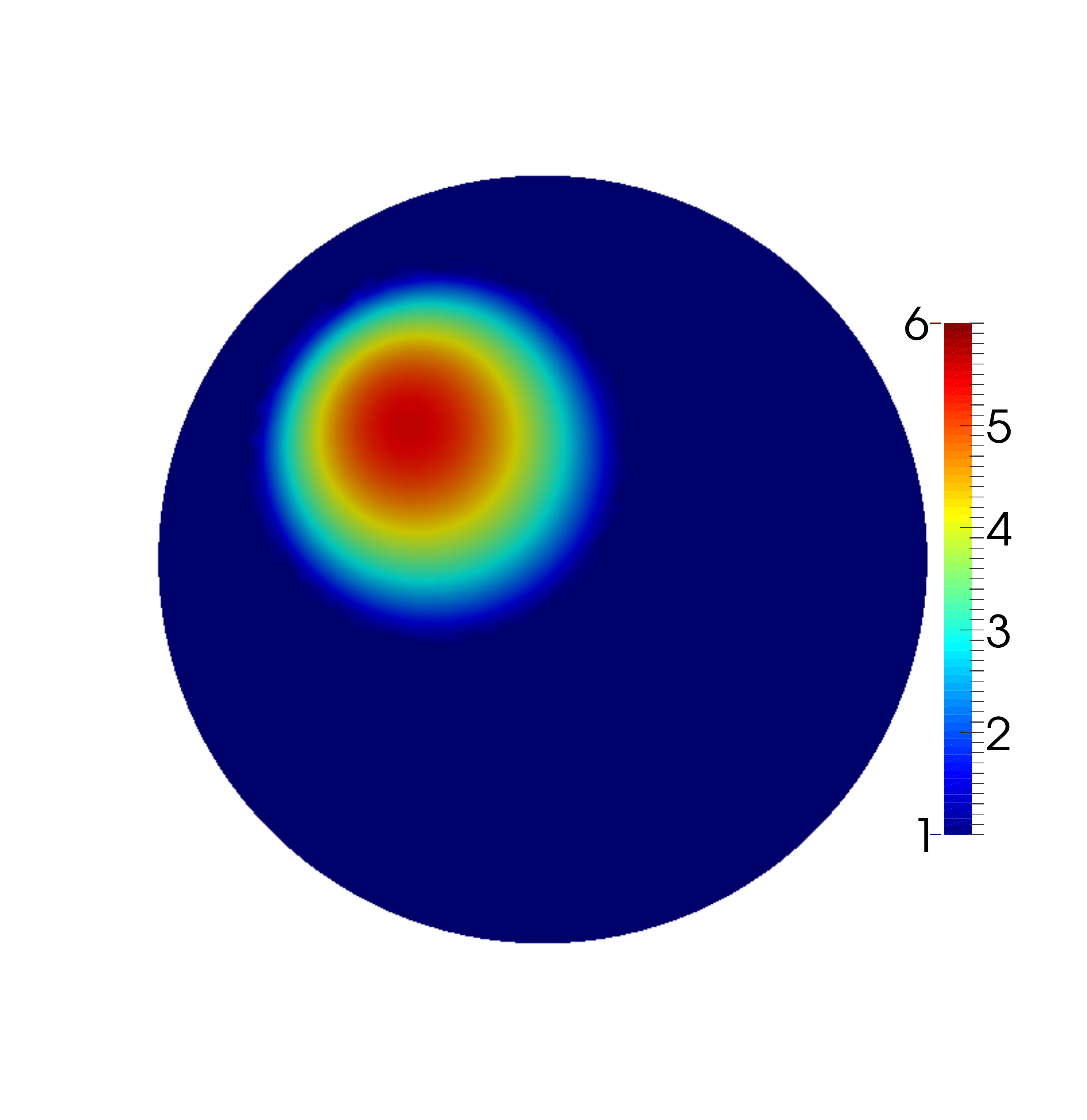}
	\caption{$\alpha = 10^{-3}$}
\end{subfigure}%
\hfill
\begin{subfigure}[h]{.33\linewidth}
	\includegraphics[width = \linewidth, trim = 4cm 4cm 1cm 4cm, clip=true]{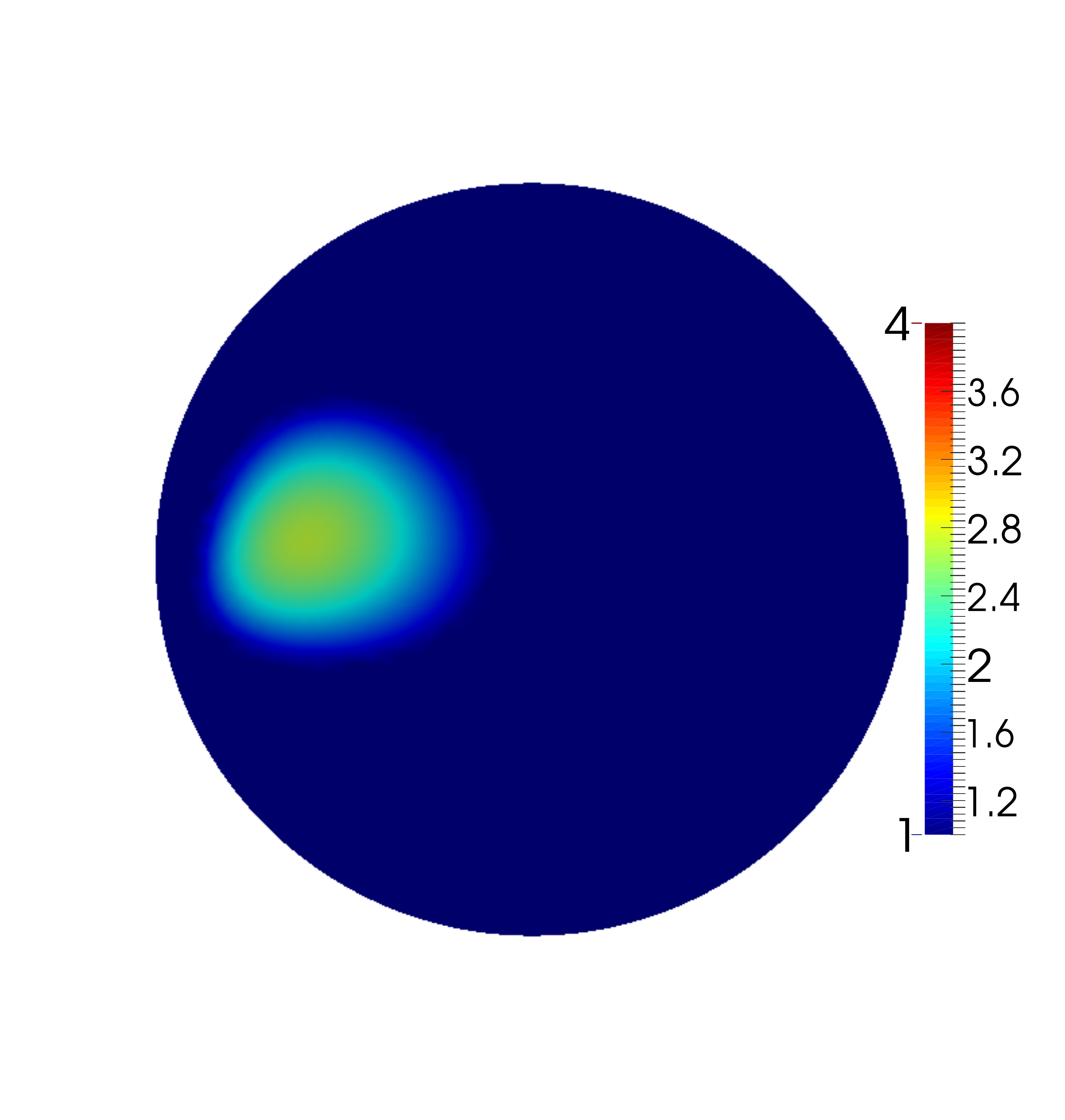}
	\caption{$\alpha = 5\cdot 10^{-4}$}
\end{subfigure}%
\hfill
\begin{subfigure}[h]{.33\linewidth}
	\includegraphics[width = \linewidth, trim = 4cm 4cm 1cm 4cm, clip=true]{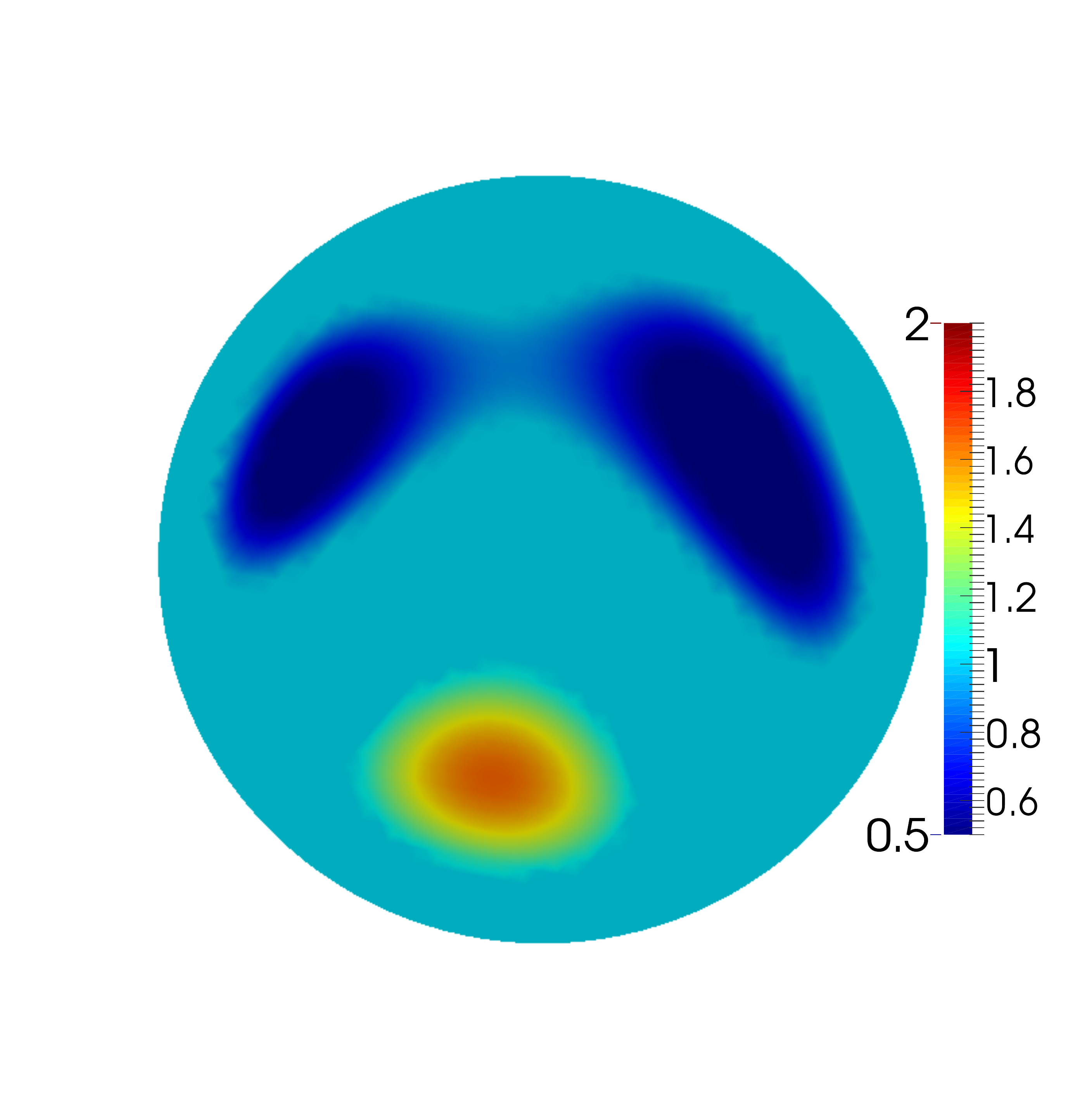}
	\caption{$\alpha = 6.5\cdot 10^{-4}$}
\end{subfigure}% 
\caption{Sparse reconstruction of the phantoms in \fref{fig:phantoms}.}
\label{fig:recon1}
\end{figure}

For the kite-shaped phantom we only get what seems like a convex approximation of the shape. It is seen in \cite{jin2012} that the algorithm is able to reconstruct some types of non-convex inclusions such as the hole in a ring-shaped phantom, however those inclusions are much larger which makes it easier to distinguish from similar convex inclusions. 

We note that the method is very stable towards noise. In \fref{fig:recon2} it is shown how unreasonable amounts of noise only leads to small deformations in the shape of the reconstructed inclusion.

\begin{figure}[h]%
\centering
\begin{subfigure}[b]{.4\linewidth}
	\includegraphics[width = \linewidth]{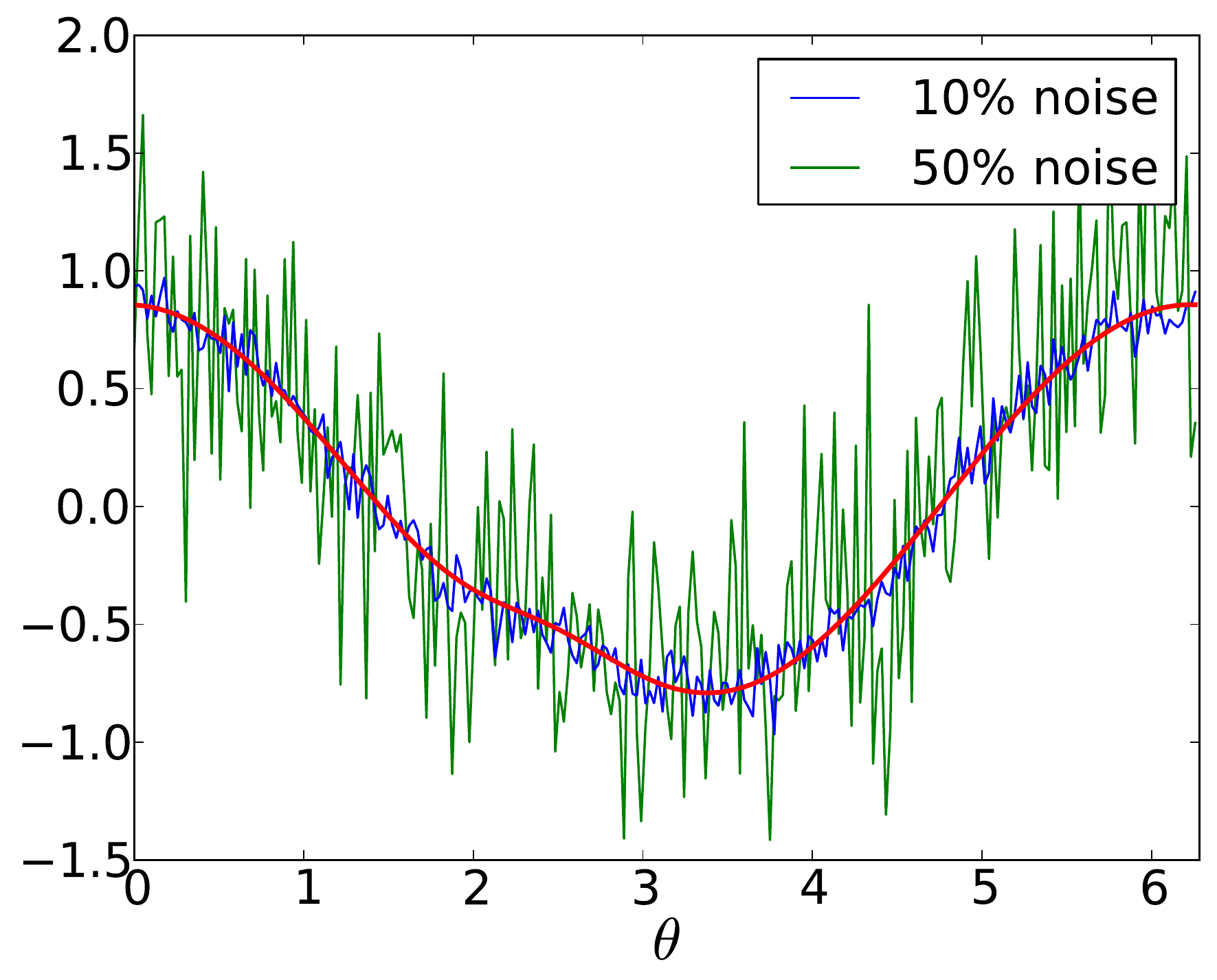}
\end{subfigure}%
\hfill
\begin{subfigure}[b]{.3\linewidth}
	\includegraphics[width = \linewidth, trim = 4cm 4cm 1cm 4cm, clip=true]{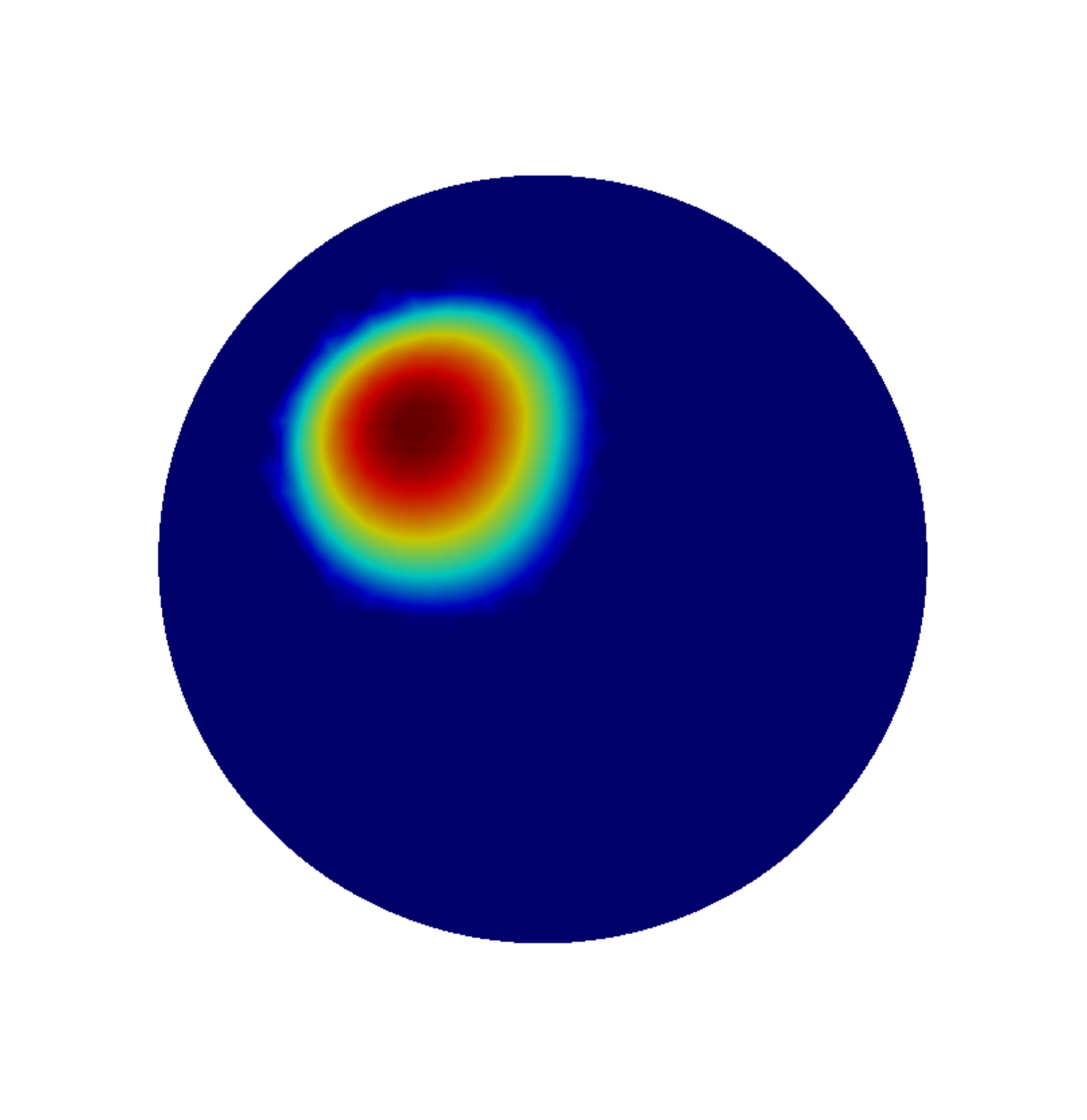}
	\caption*{10\% noise}
\end{subfigure}%
\hfill
\begin{subfigure}[b]{.3\linewidth}
	\includegraphics[width = \linewidth, trim = 4cm 4cm 1cm 4cm, clip=true]{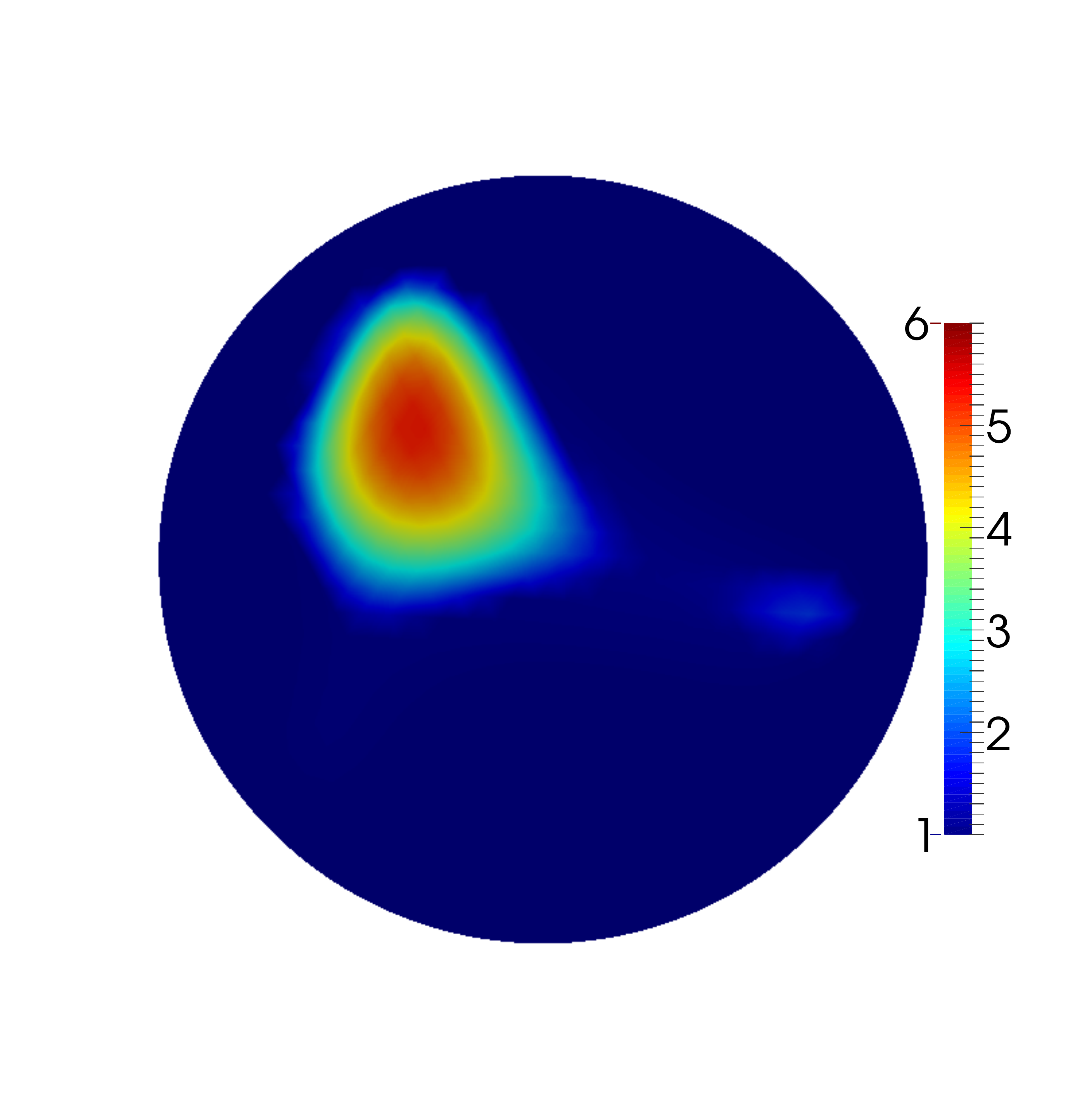}
	\caption*{50\% noise}
\end{subfigure}% 
\caption{\textbf{Left:} Dirichlet data corresponding to $g = \cos(\theta)$ for the phantom in \fref{fig:circphantom}, with 10\% and 50\% noise level. \textbf{Middle:} reconstruction for 10\% noise level. \textbf{Right:} reconstruction for 50\% noise level.}
\label{fig:recon2}
\end{figure}

In order to investigate the use of prior information we consider the phantom in \fref{fig:circphantom}, and let $B(r)$ denote a ball centered at the correct inclusion and with radius $r$. Now we can investigate reconstructions with prior information assuming that the support of $\dsr$ is $\overline{B((1+\delta r)r^*)}$ for $r^*$ being the correct radius of the inclusion. \Fref{fig:r-dep} shows that underestimating the support of the inclusion $\delta r < 0$ is heavily enforced, and the contrast is vastly overestimated in the reconstruction as shown in \fref{fig:meanmax} (note that this can not be seen in \fref{fig:r-dep} as the color scale for the phantom is applied). 

\begin{figure}[h]%
\centering
\begin{subfigure}[b]{.33\linewidth}
	\includegraphics[width = \linewidth, trim = 4cm 4cm 1cm 4cm, clip=true]{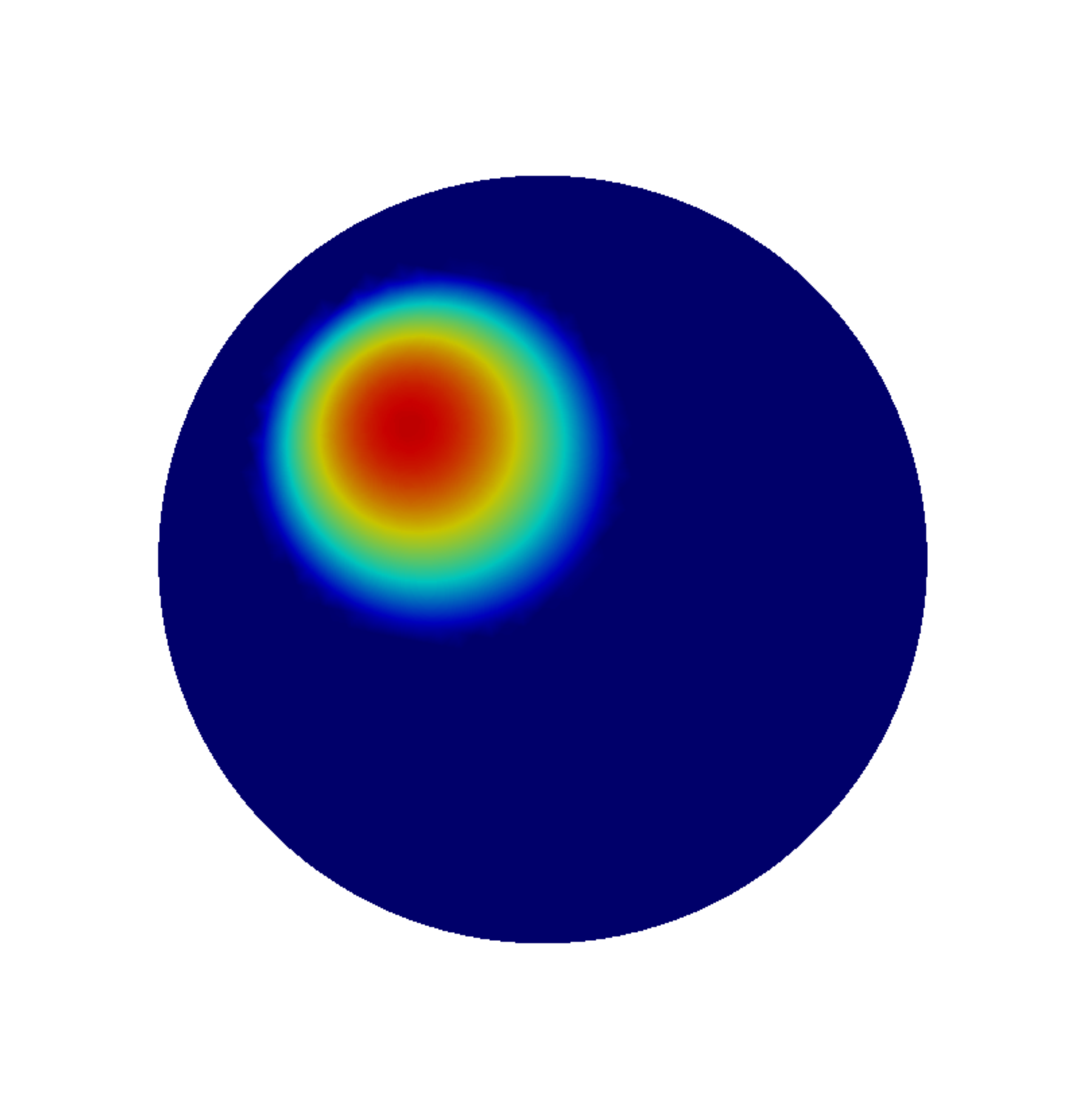}
	\caption*{No prior}
\end{subfigure}%
\hfill
\begin{subfigure}[b]{.33\linewidth}
	\includegraphics[width = \linewidth, trim = 4cm 4cm 1cm 4cm, clip=true]{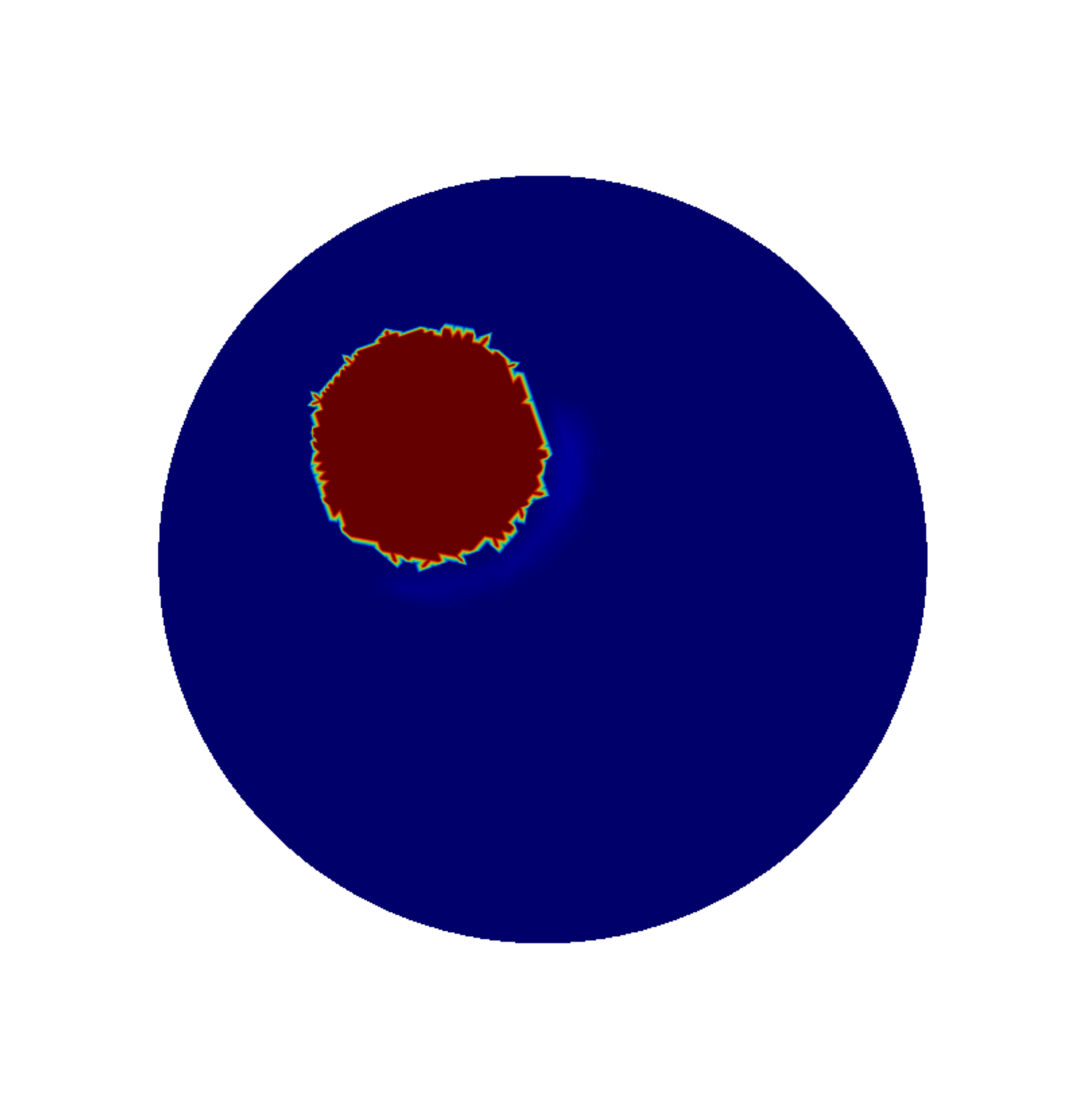}
	\caption*{$\delta r = -0.25$}
\end{subfigure}%
\hfill
\begin{subfigure}[b]{.33\linewidth}
	\includegraphics[width = \linewidth, trim = 4cm 4cm 1cm 4cm, clip=true]{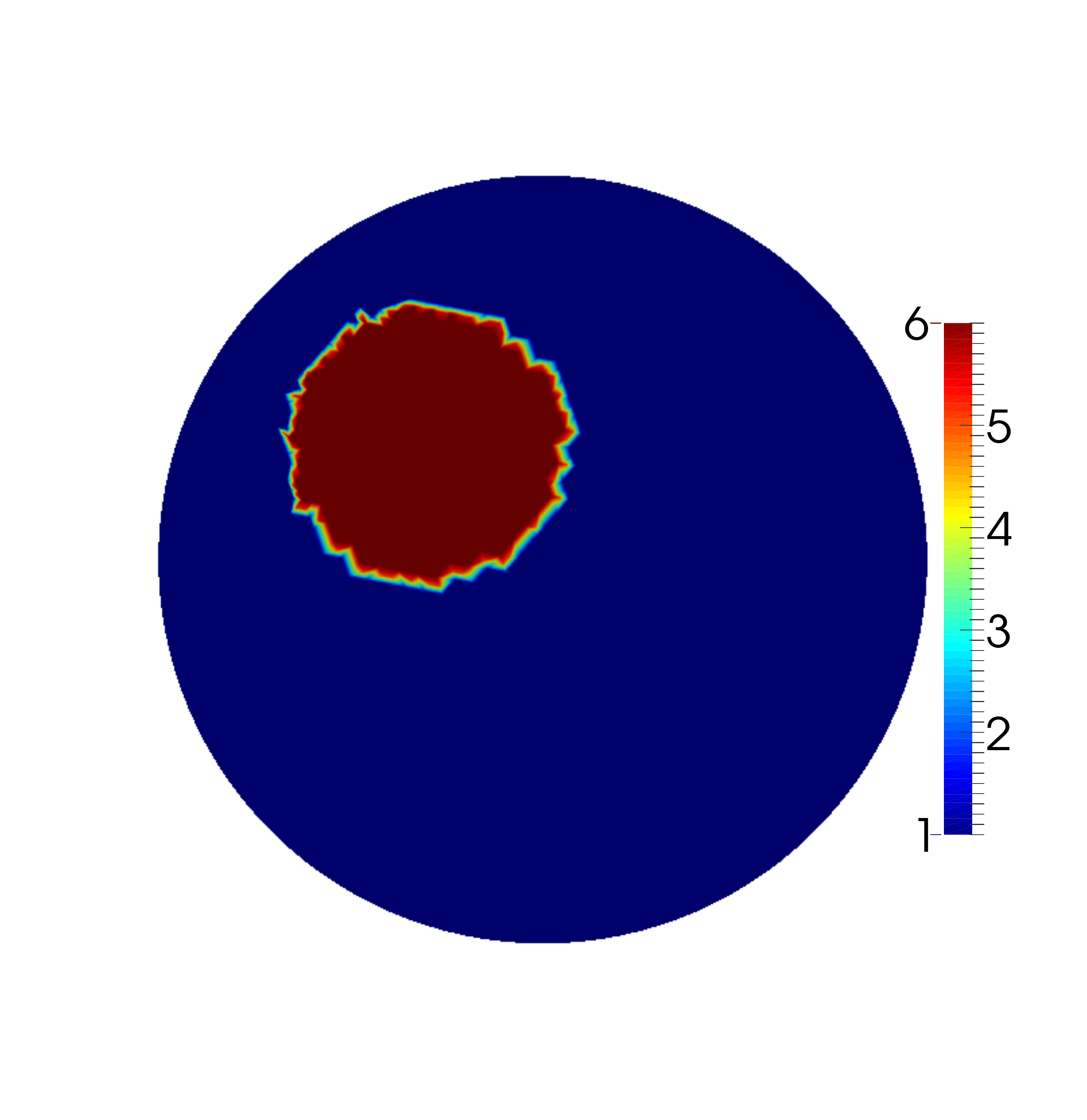}
	\caption*{$\delta r = -0.10$}
\end{subfigure}% 
\\[5mm]
\begin{subfigure}[b]{.33\linewidth}
	\includegraphics[width = \linewidth, trim = 4cm 4cm 1cm 4cm, clip=true]{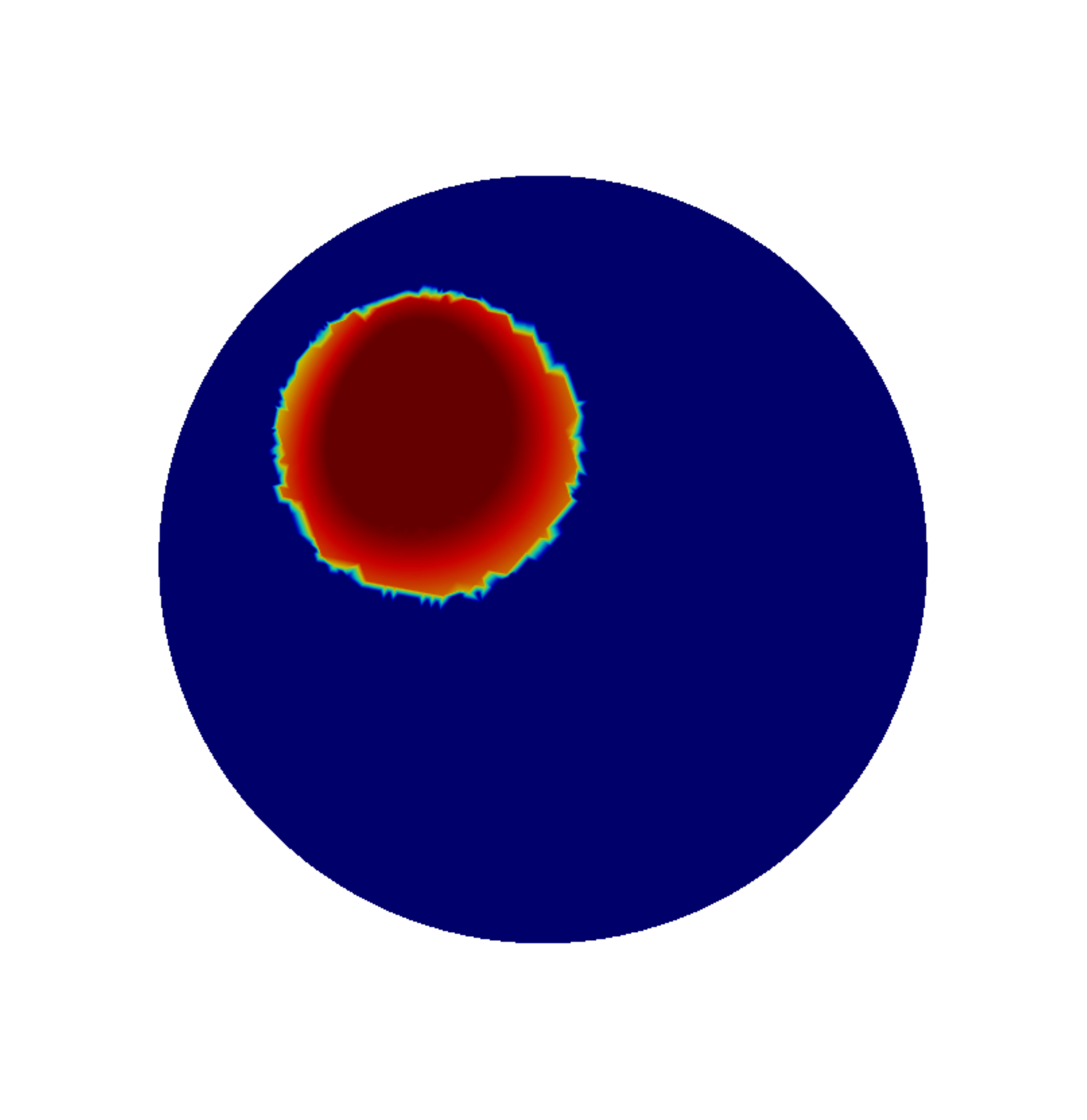}
	\caption*{$\delta r = 0$}
\end{subfigure}%
\hfill
\begin{subfigure}[b]{.33\linewidth}
	\includegraphics[width = \linewidth, trim = 4cm 4cm 1cm 4cm, clip=true]{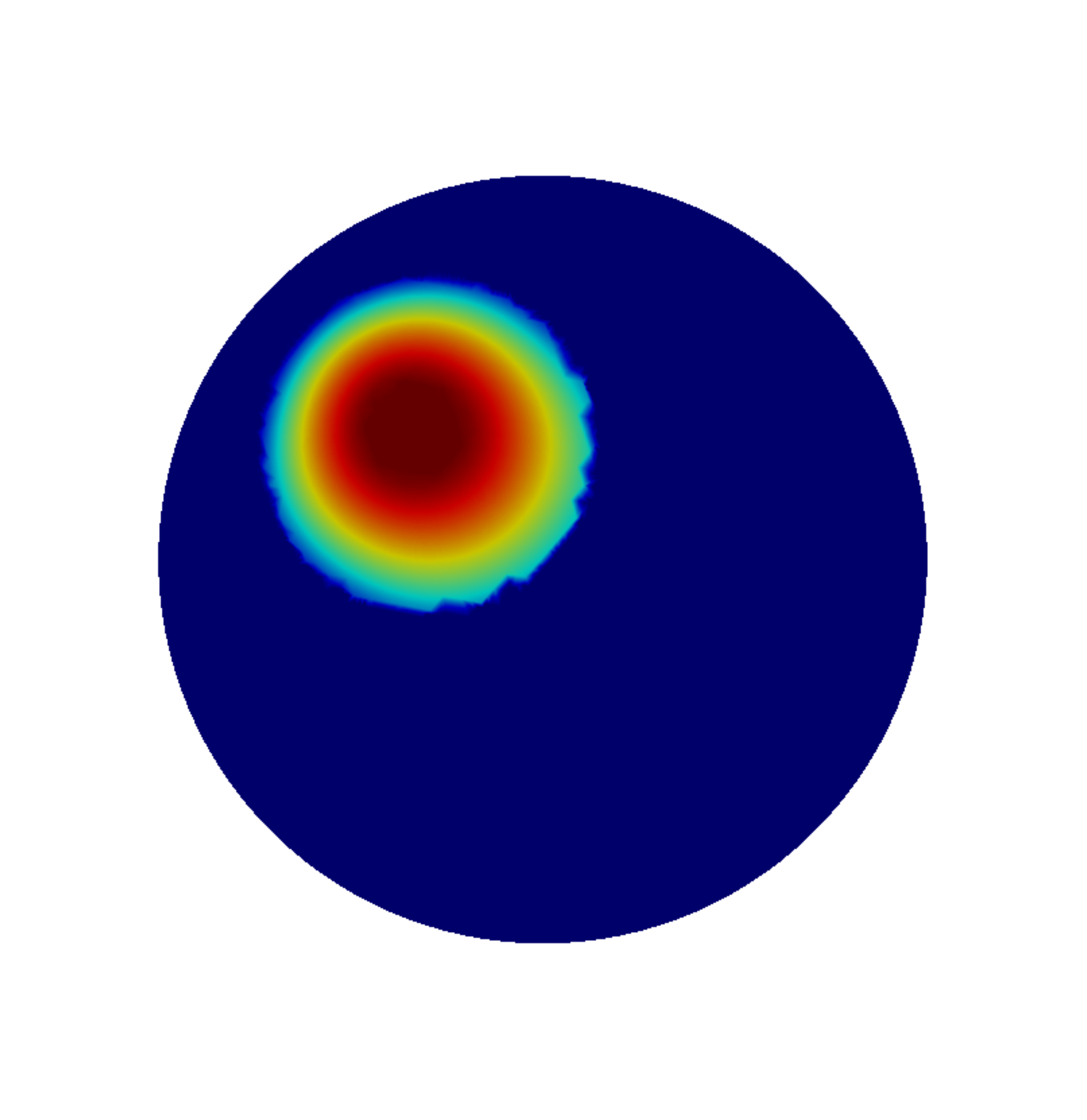}
	\caption*{$\delta r = 0.10$}
\end{subfigure}% 
\hfill
\begin{subfigure}[b]{.33\linewidth}
	\includegraphics[width = \linewidth, trim = 4cm 4cm 1cm 4cm, clip=true]{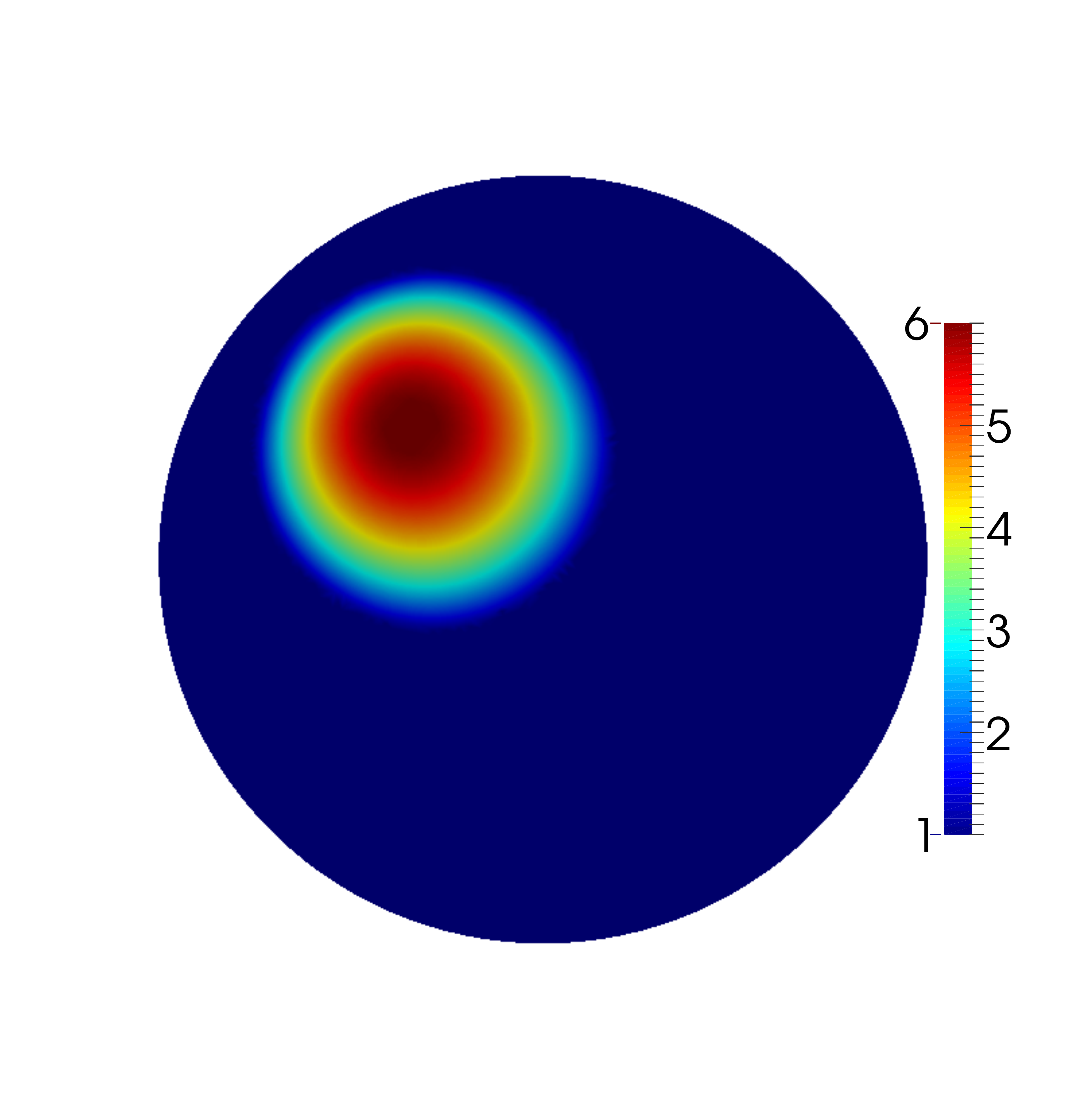}
	\caption*{$\delta r = 0.25$}
\end{subfigure}% 
\caption{Sparse reconstruction of the phantom in \fref{fig:circphantom} for varying $\delta r$. The colorbar is truncated at $[1,6]$.}
\label{fig:r-dep}
\end{figure}

Interestingly, when overestimating the support, the contrast and support of the reconstructed inclusion does not suffer particularly. Intuitively, this corresponds to increasing $\delta r$ such that the assumed support of $\dsr$ contains the entire domain $\Omega$, which corresponds to the case with no prior information. For a subset $E\subseteq\Omega$ denote by $\sigma_E \equiv \absm{E}^{-1}\int_E \sigma dx$ the average of $\sigma$ on $E$, and denote by $\sigma_{\max} \equiv \max_{j}\absm{\sigma(x_j)}$ the maximum of $\sigma$ on the mesh nodes. Then \fref{fig:meanmax} gives a good indication of the aforementioned intuition, where around $\delta r = 0$ both $\sigma_B$ and $\sigma_{\max}$ levels off around the correct contrast of the inclusion (the red line) and stays there for $\delta r > 0$. It should be noted that even a 25\% overestimation of the support leads to a better contrast in the reconstruction than if no prior information was applied, as seen in \fref{fig:r-dep}.

Having an overestimation of the support for $\dsr$ also seems to be a reasonable assumption. Definitely there is the case of no prior information which means that $\supp \dsr$ is assumed to be $\Omega$. If the estimation comes from another method such as total variation regularization, then the support is typically slightly overestimated while the contrast suffers \cite{Borsic}. Thus we can use the overestimated support to get a good localisation and contrast reconstruction simultaneously.

\begin{figure}[h]%
\centering
\includegraphics[width = 0.6 \linewidth]{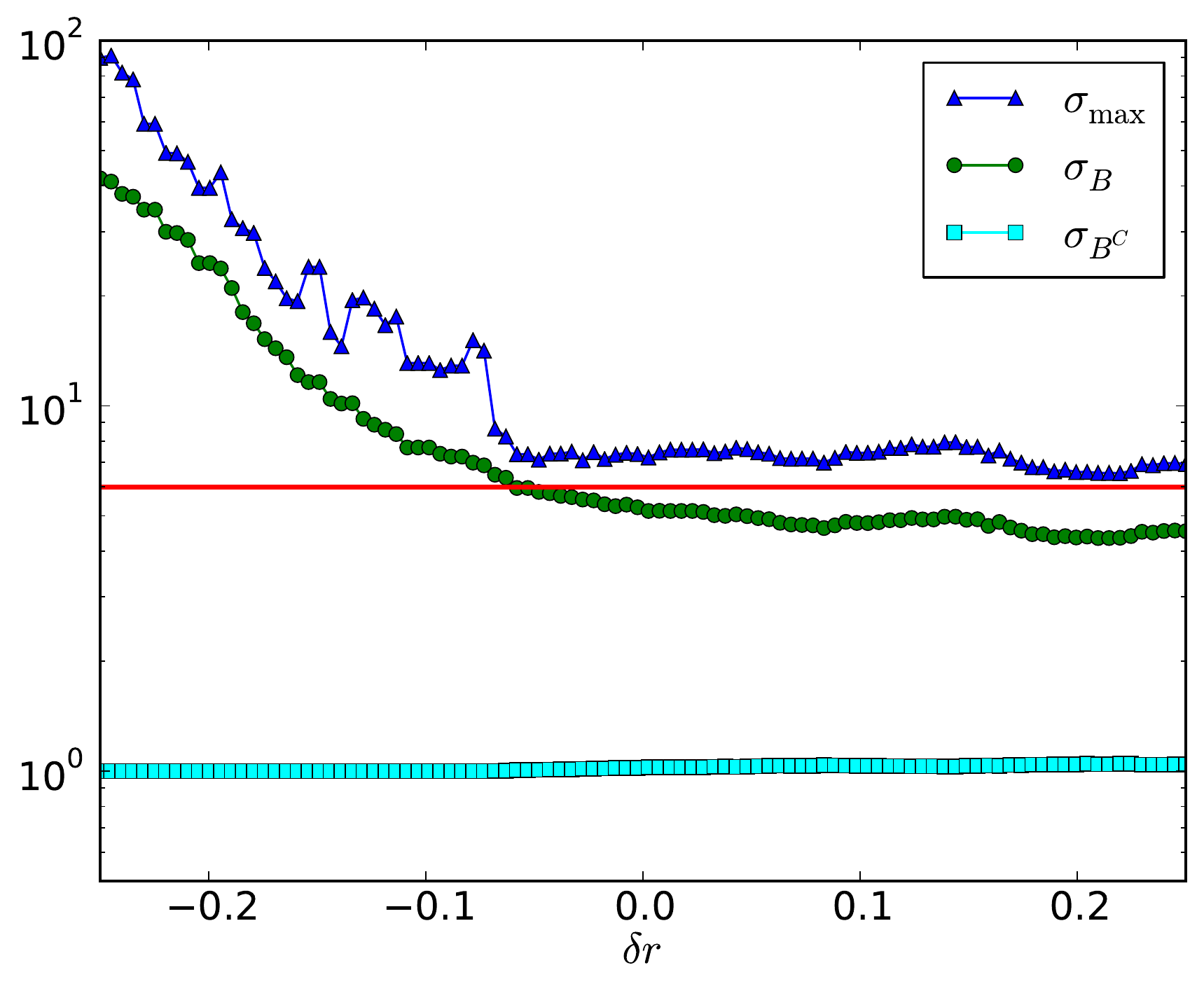}
\caption{Behaviour of sparsity reconstruction based on the phantom in \fref{fig:circphantom} for varying $\delta r$.}
\label{fig:meanmax}
\end{figure}

\Fref{fig:kitepriorrecon} shows how the reconstruction of the kite-shaped phantom can be vastly improved. Note that not only is $\supp \dsr$ better approximated, but the contrast is also highly improved. It is not surprising that we can achieve an almost perfect reconstruction if $\supp \dsr$ is exactly known, however it is a good benchmark to compare the cases for the overestimated support as it shows how well the method can possibly do. 

\begin{figure}[h]%
\centering
\begin{subfigure}[b]{.33\linewidth}
	\includegraphics[width = \linewidth, trim = 4cm 4cm 1cm 4cm, clip=true]{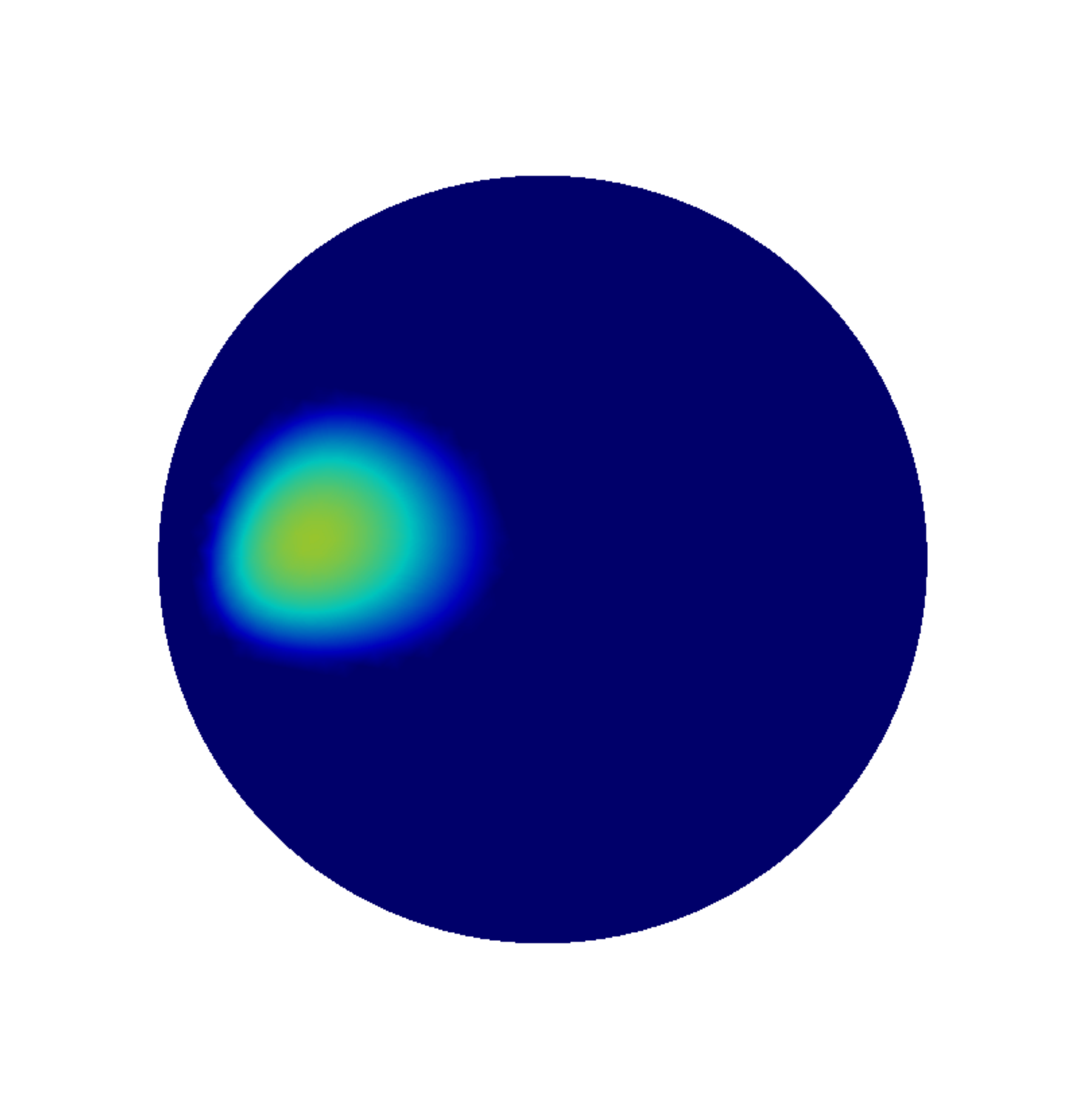}
	\caption*{No prior}
\end{subfigure}%
\hfill
\begin{subfigure}[b]{.33\linewidth}
	\includegraphics[width = \linewidth, trim = 4cm 4cm 1cm 4cm, clip=true]{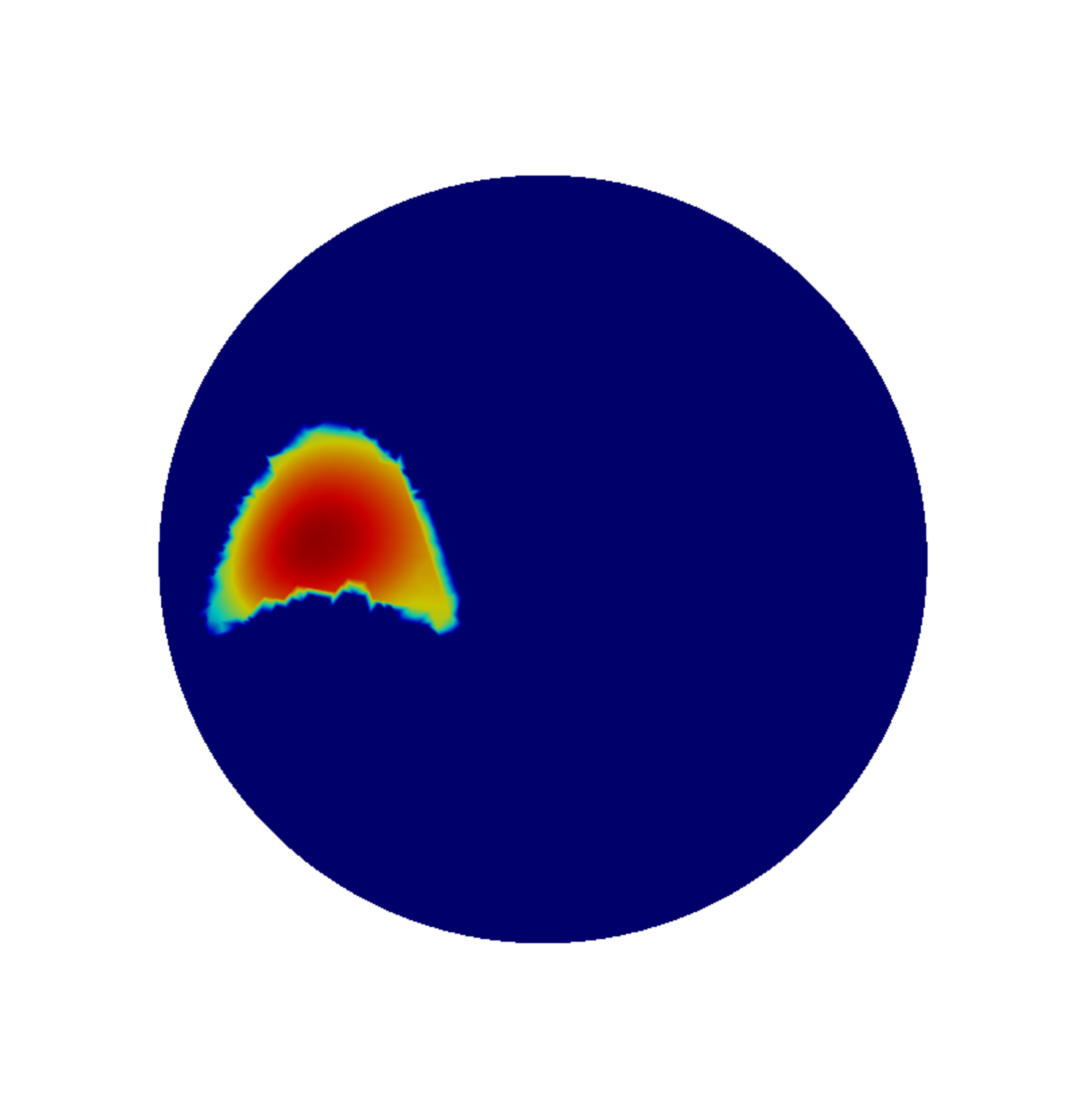}
	\caption*{10\% overestimated support}
\end{subfigure}%
\hfill
\begin{subfigure}[b]{.33\linewidth}
	\includegraphics[width = \linewidth, trim = 4cm 4cm 1cm 4cm, clip=true]{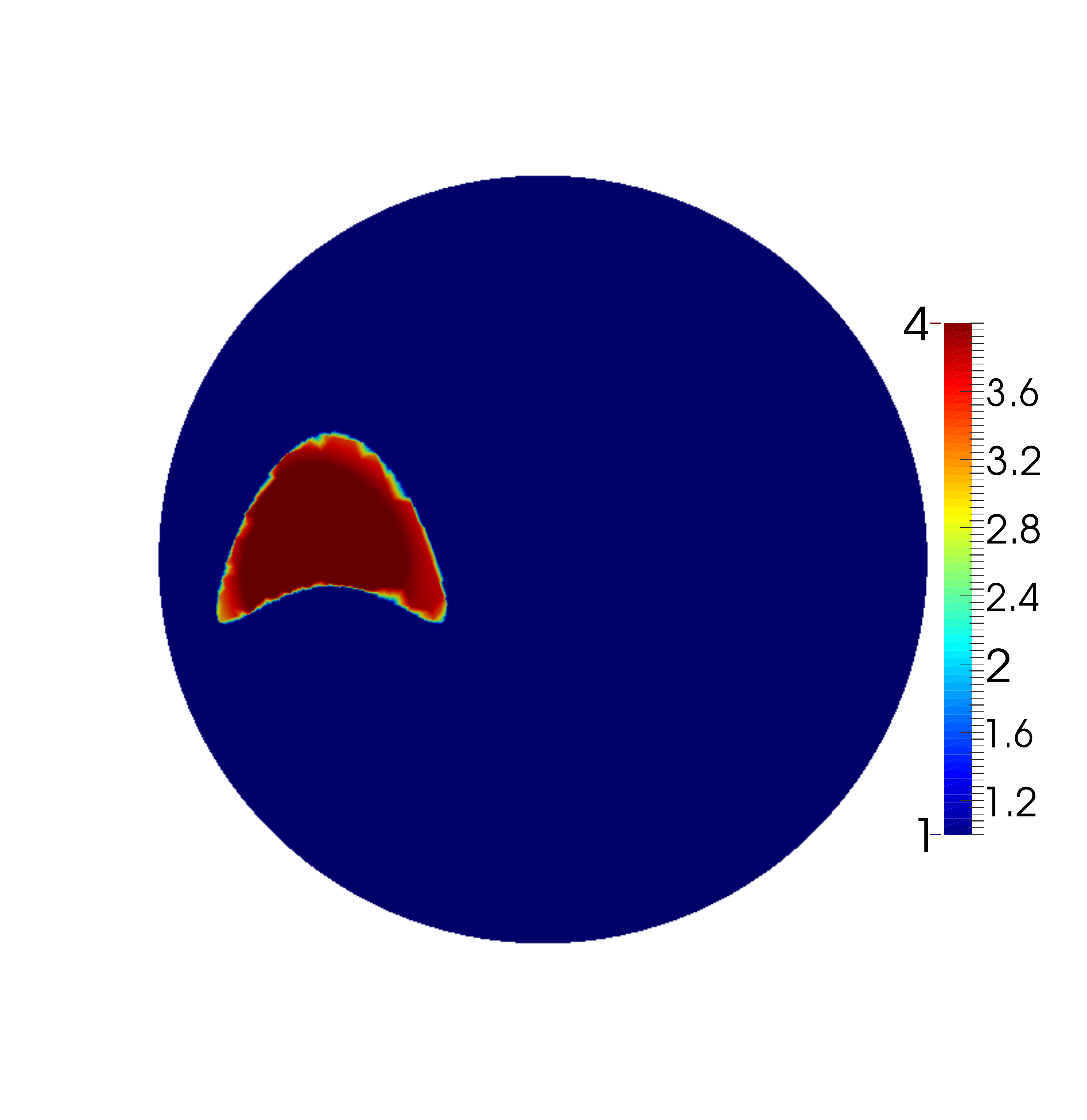}
	\caption*{Exact support}
\end{subfigure}% 
\caption{Sparse reconstruction of the phantom in \fref{fig:kitephantom}.}
\label{fig:kitepriorrecon}
\end{figure}

\subsection{Partial Boundary Data}

For the partial data problem we choose $\Gamma = \gamD = \gamN = \{\theta\in(\theta_1,\theta_2)\}$ for $0\leq \theta_1 < \theta_2 \leq 2\pi$.

\begin{figure}[h]%
\centering
\begin{subfigure}[b]{.33\linewidth}
	\includegraphics[width = \linewidth, trim = 3.6cm 4cm 0cm 4cm, clip=true]{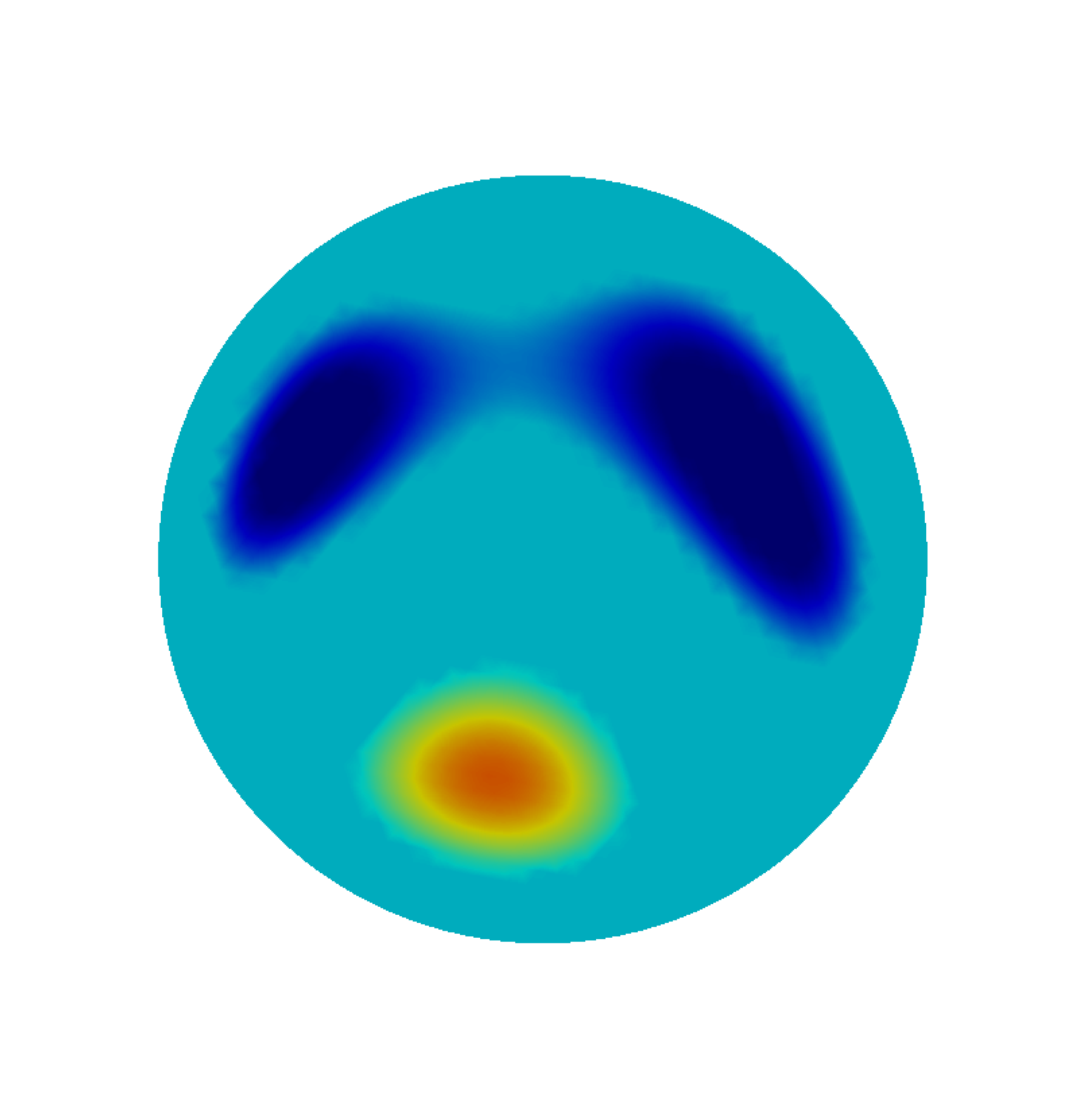}
	\caption*{Full data}
\end{subfigure}%
\hfill
\begin{subfigure}[b]{.33\linewidth}
\includegraphics[width = \linewidth, trim = 3.6cm 4cm 0cm 4cm, clip=true]{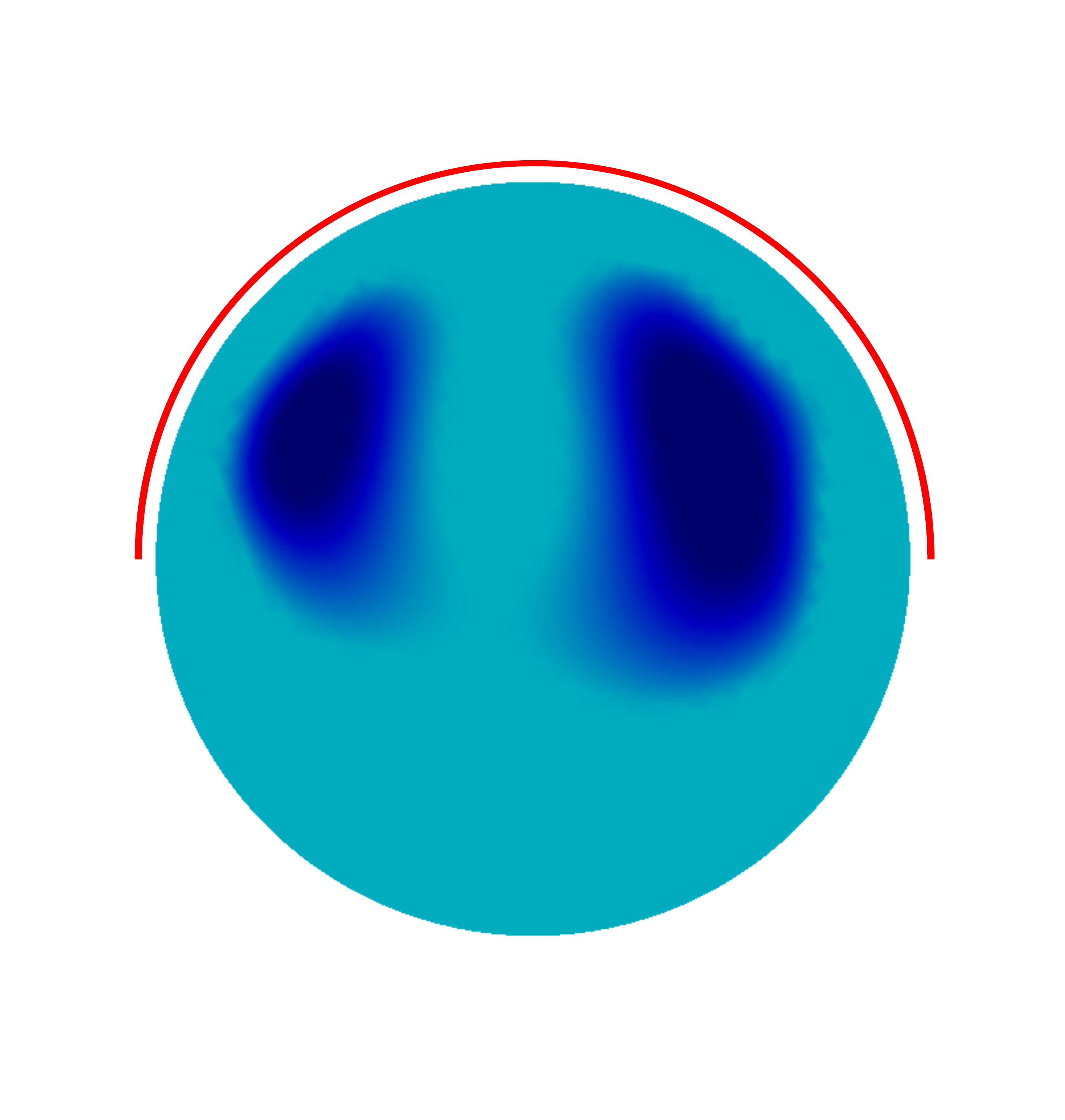}
\caption*{Top half}
\end{subfigure}%
\hfill
\begin{subfigure}[b]{.33\linewidth}
\includegraphics[width = \linewidth, trim = 3.6cm 4cm 0cm 4cm, clip=true]{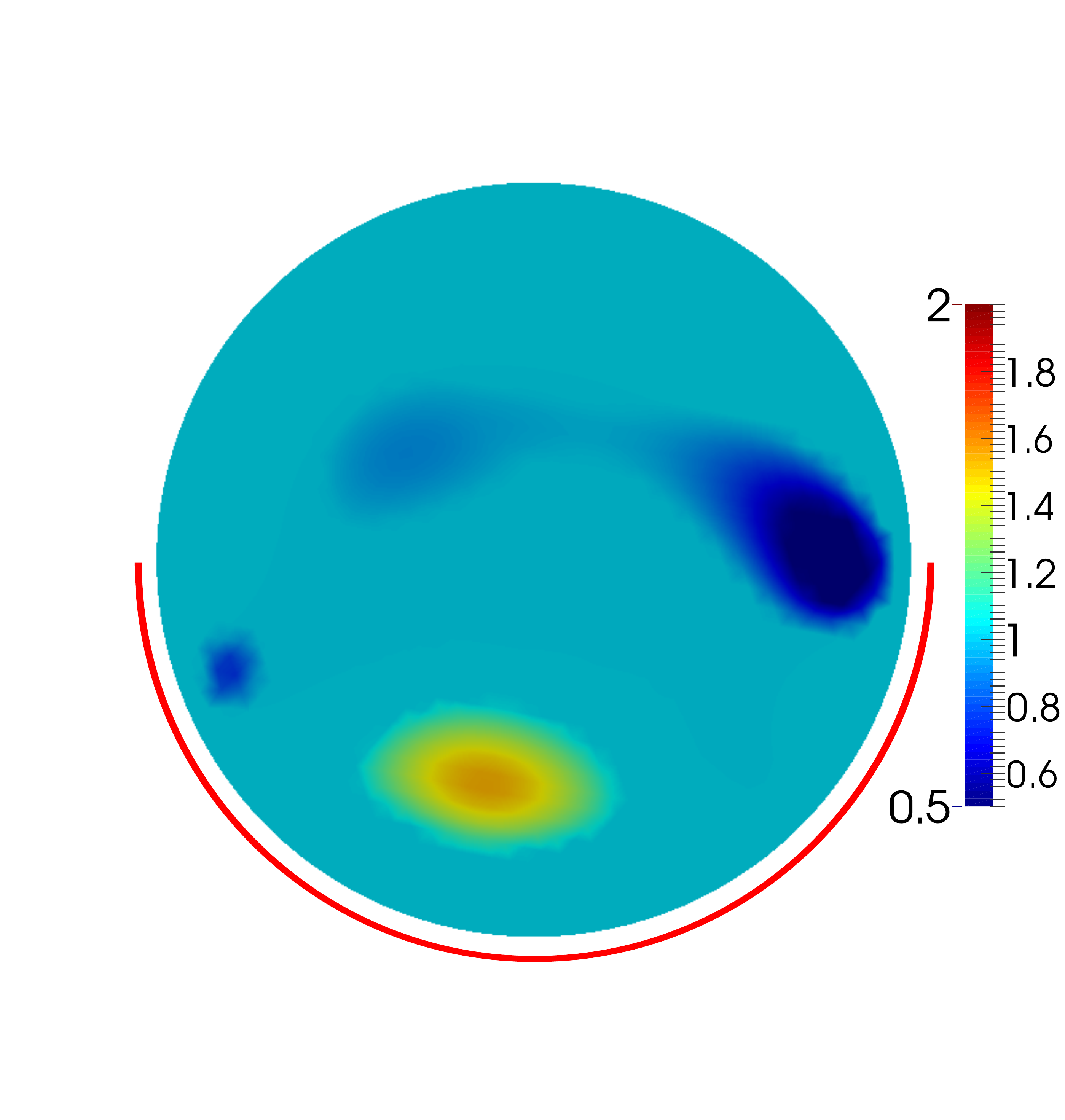}
	\caption*{Bottom half}
\end{subfigure}% 
\caption{Sparse reconstruction of the phantom in \fref{fig:hnlphantom}. \textbf{Left:} $\Gamma=\po$. \textbf{Middle:} $(\theta_1,\theta_2) = (0,\pi)$. \textbf{Right:} $(\theta_1,\theta_2) = (\pi,2\pi)$.}
\label{fig:hnlpartial}
\end{figure}

In \fref{fig:hnlpartial} we observe that with data on the top half of the unit circle it is actually possible to get very good contrast and also reasonable localization of the two large inclusions. There is still a clear separation of the inclusions, while the small inclusion is not reconstructed at all. With data on the bottom half the small inclusion is reconstructed almost as well as with full boundary data, but the larger inclusions are only vaguely visible. This is the kind of behaviour that is expected from partial data EIT, and in practice it implies that we can only expect reasonable reconstruction close to where the measurements are taken. 

\begin{figure}[h]%
\centering
\begin{subfigure}[b]{.25\linewidth}
\includegraphics[width = \linewidth, trim = 3.5cm 4cm 0cm 4cm, clip=true]{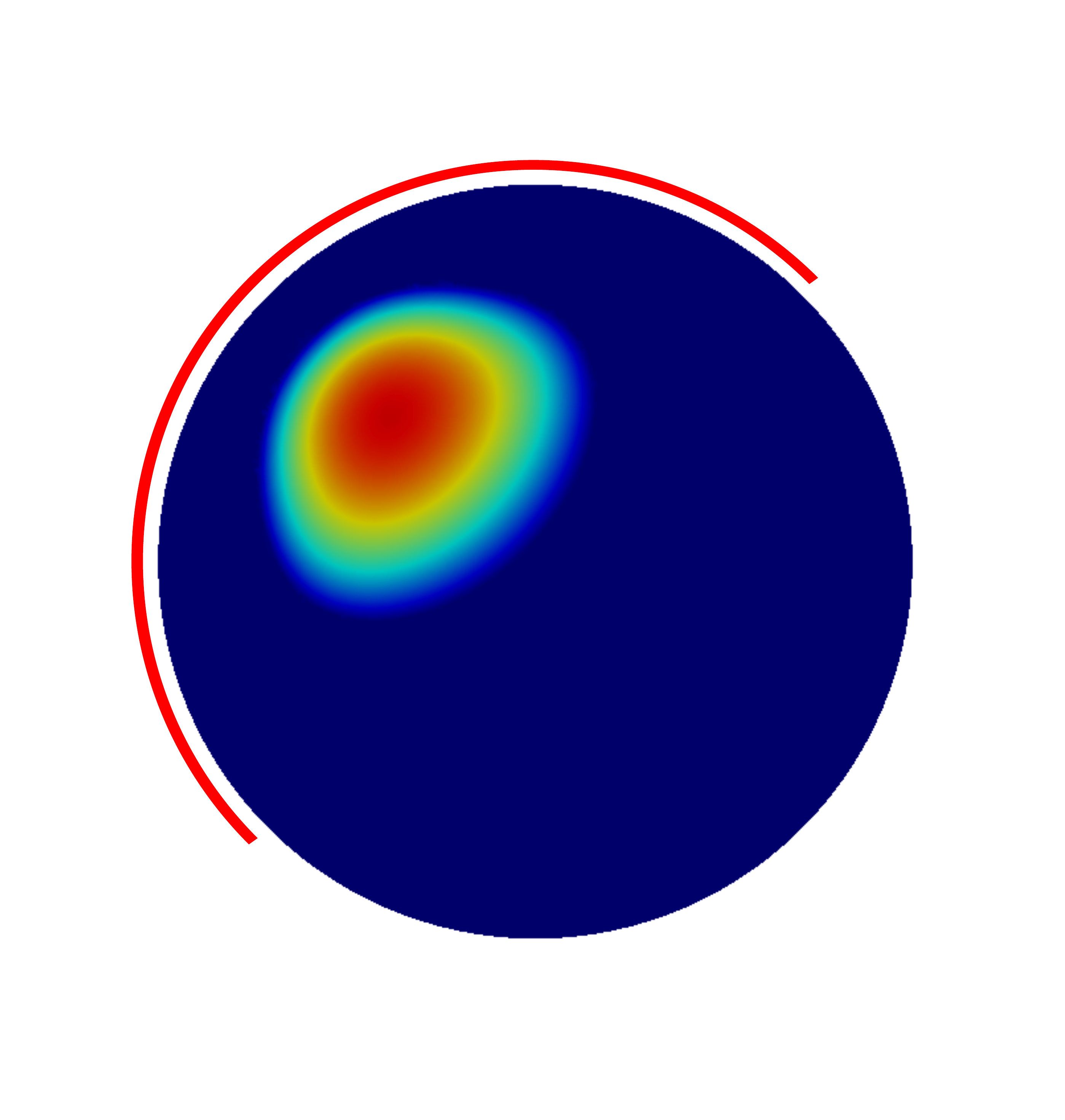}
	\caption{{}}
\end{subfigure}%
\hfill
\begin{subfigure}[b]{.25\linewidth}
\includegraphics[width = \linewidth, trim = 3.5cm 4cm 0cm 4cm, clip=true]{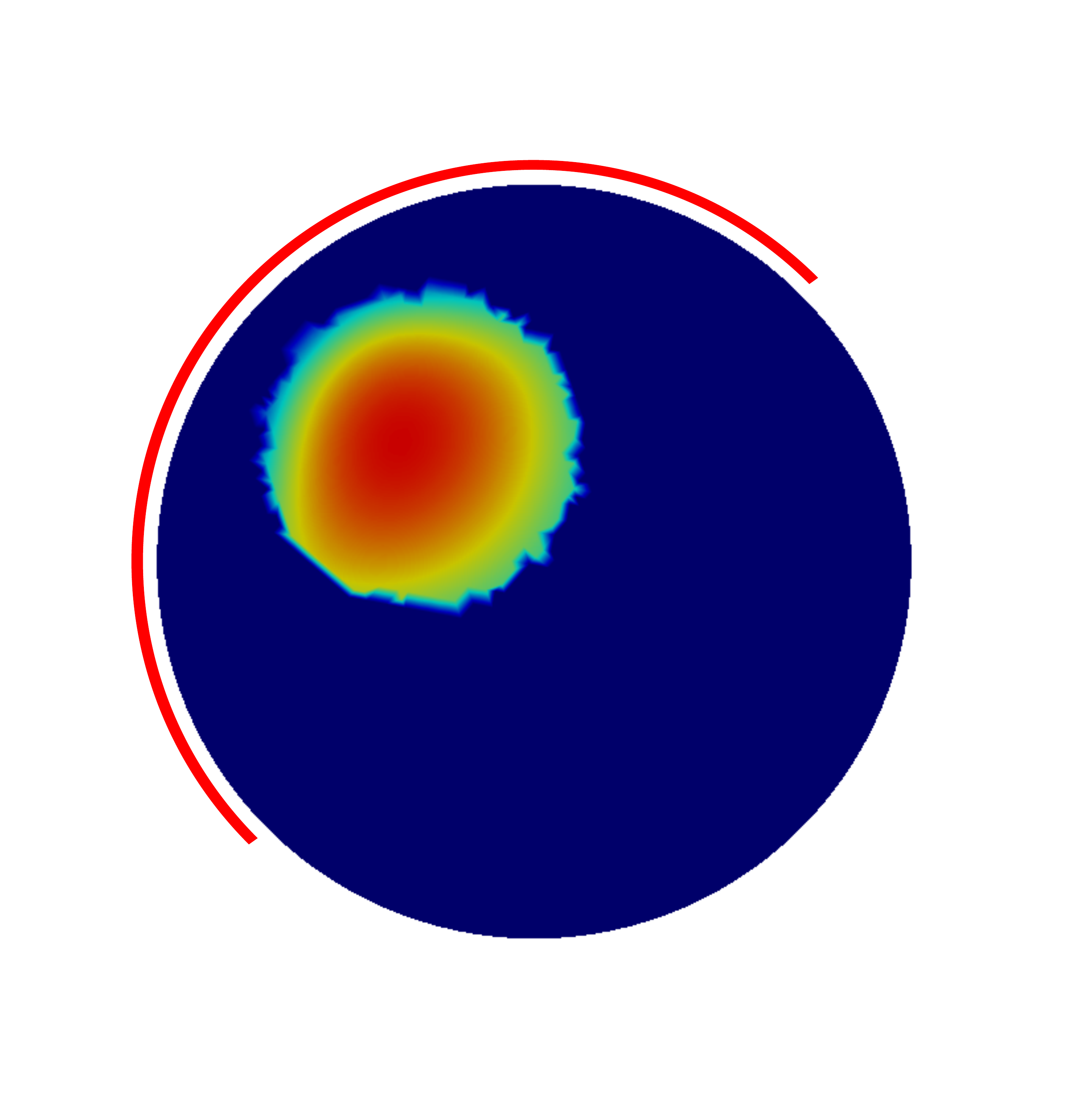}
	\caption{{}}
\end{subfigure}%
\hfill
\begin{subfigure}[b]{.25\linewidth}
\includegraphics[width = \linewidth, trim = 3.5cm 4cm 0cm 4cm, clip=true]{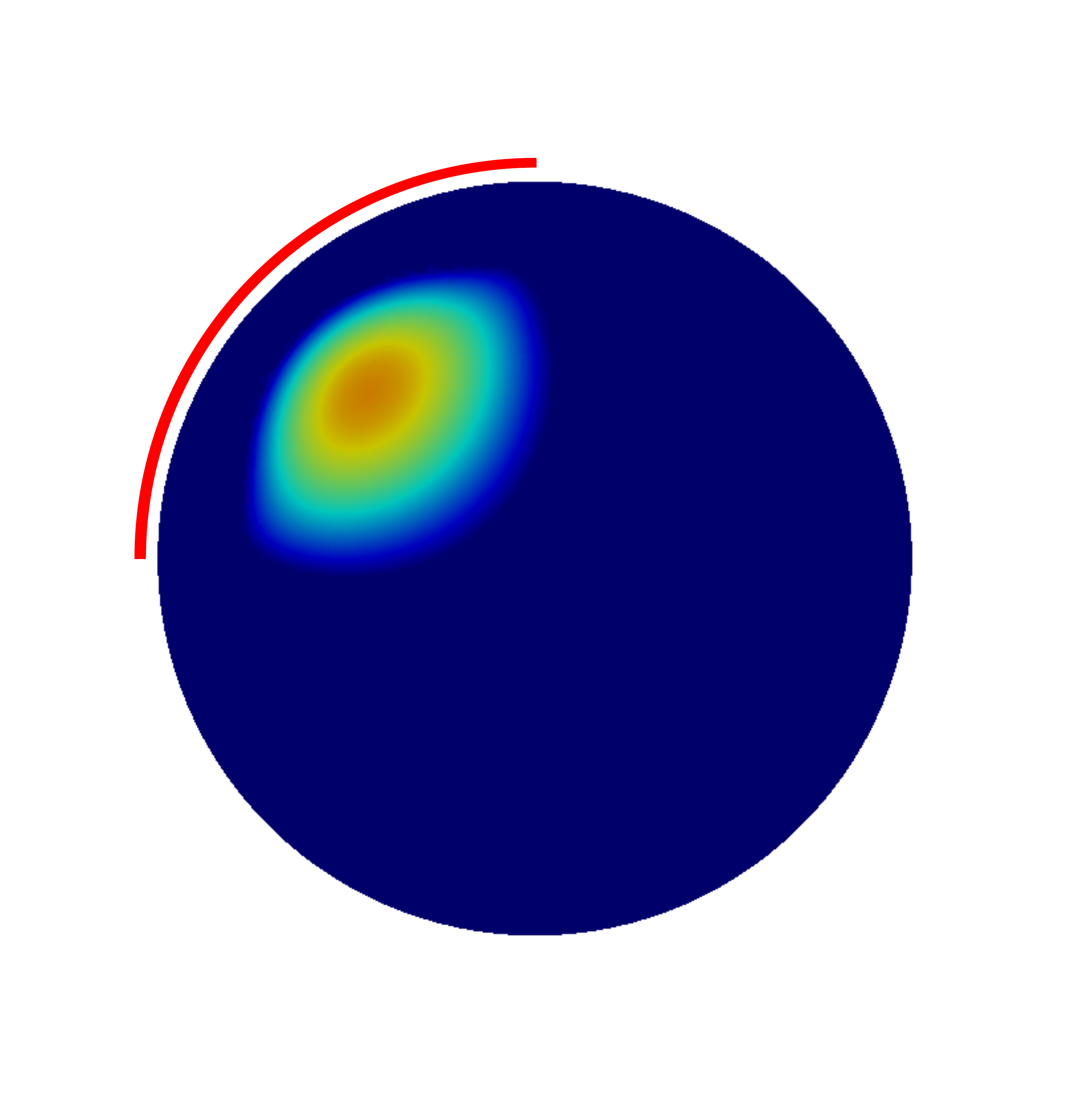}
	\caption{{}}
\end{subfigure}%
\begin{subfigure}[b]{.25\linewidth}
\includegraphics[width = \linewidth, trim = 3.5cm 4cm 0cm 4cm, clip=true]{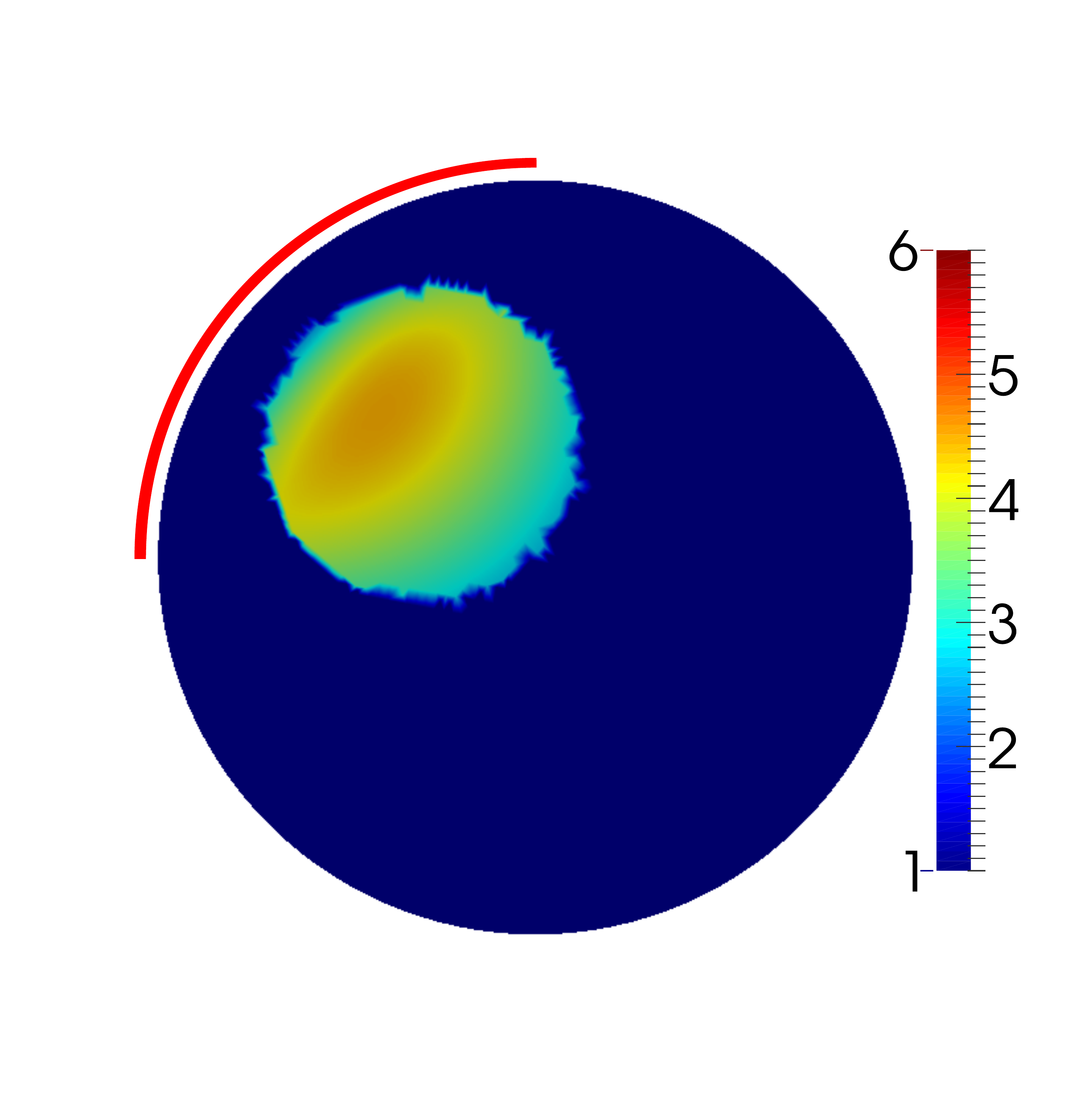}
	\caption{{}}
\end{subfigure}%
\caption{Sparse reconstruction of the phantom in \fref{fig:circphantom}. \textbf{(a):} 50\% boundary data, no prior. \textbf{(b):} 50\% boundary data with 5\% overestimated support. \textbf{(c):} 25\% boundary data, no prior. \textbf{(d):} 25\% boundary data with 5\% overestimated support.}
\label{fig:circpartial}
\end{figure}

\begin{figure}[h]%
\centering
\begin{subfigure}[b]{.25\linewidth}
\includegraphics[width = \linewidth, trim = 3.5cm 4cm 0cm 4cm, clip=true]{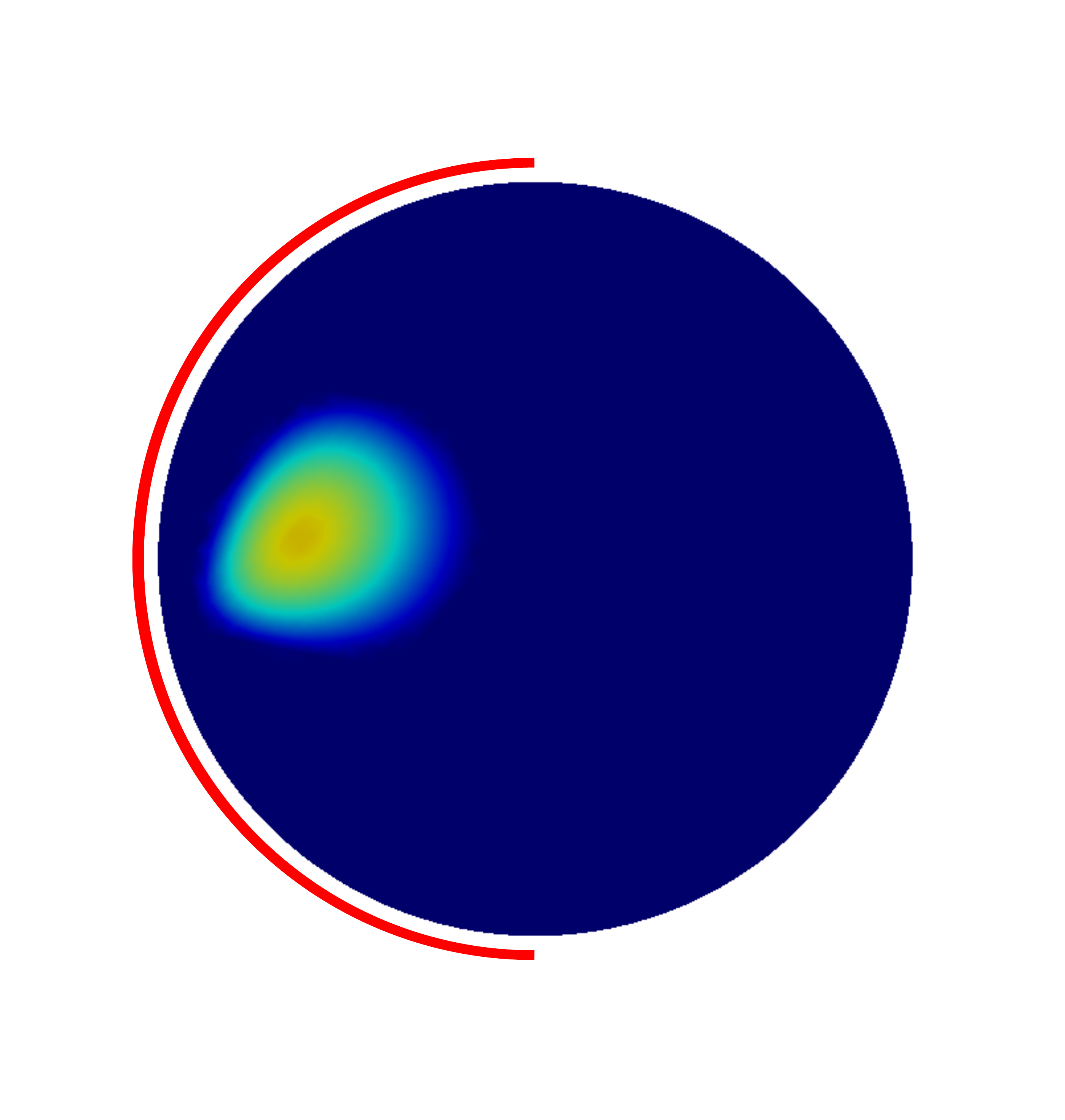}
	\caption{{}}
\end{subfigure}%
\hfill
\begin{subfigure}[b]{.25\linewidth}
\includegraphics[width = \linewidth, trim = 3.5cm 4cm 0cm 4cm, clip=true]{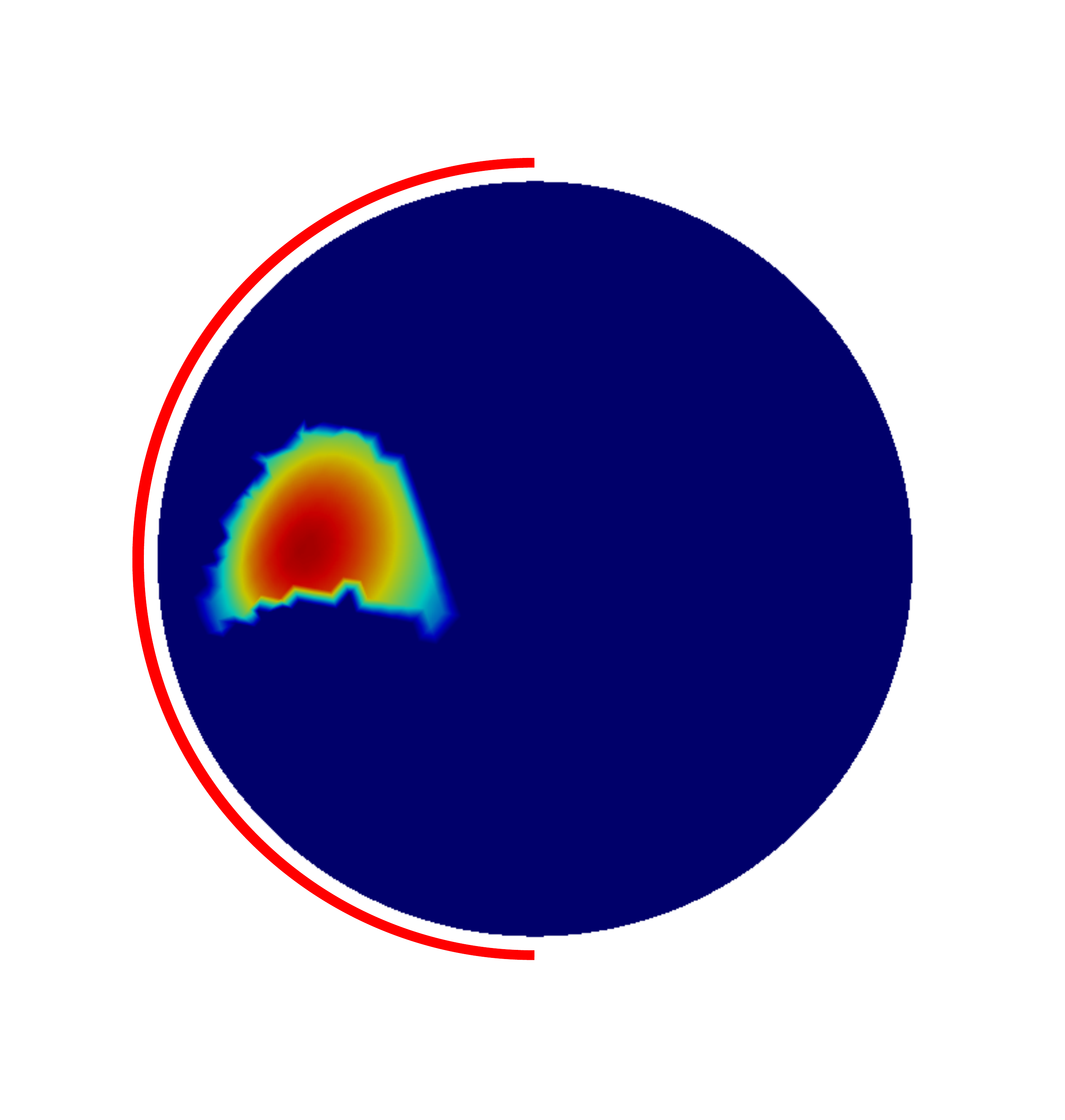}
	\caption{{}}
\end{subfigure}%
\hfill
\begin{subfigure}[b]{.25\linewidth}
\includegraphics[width = \linewidth, trim = 3.5cm 4cm 0cm 4cm, clip=true]{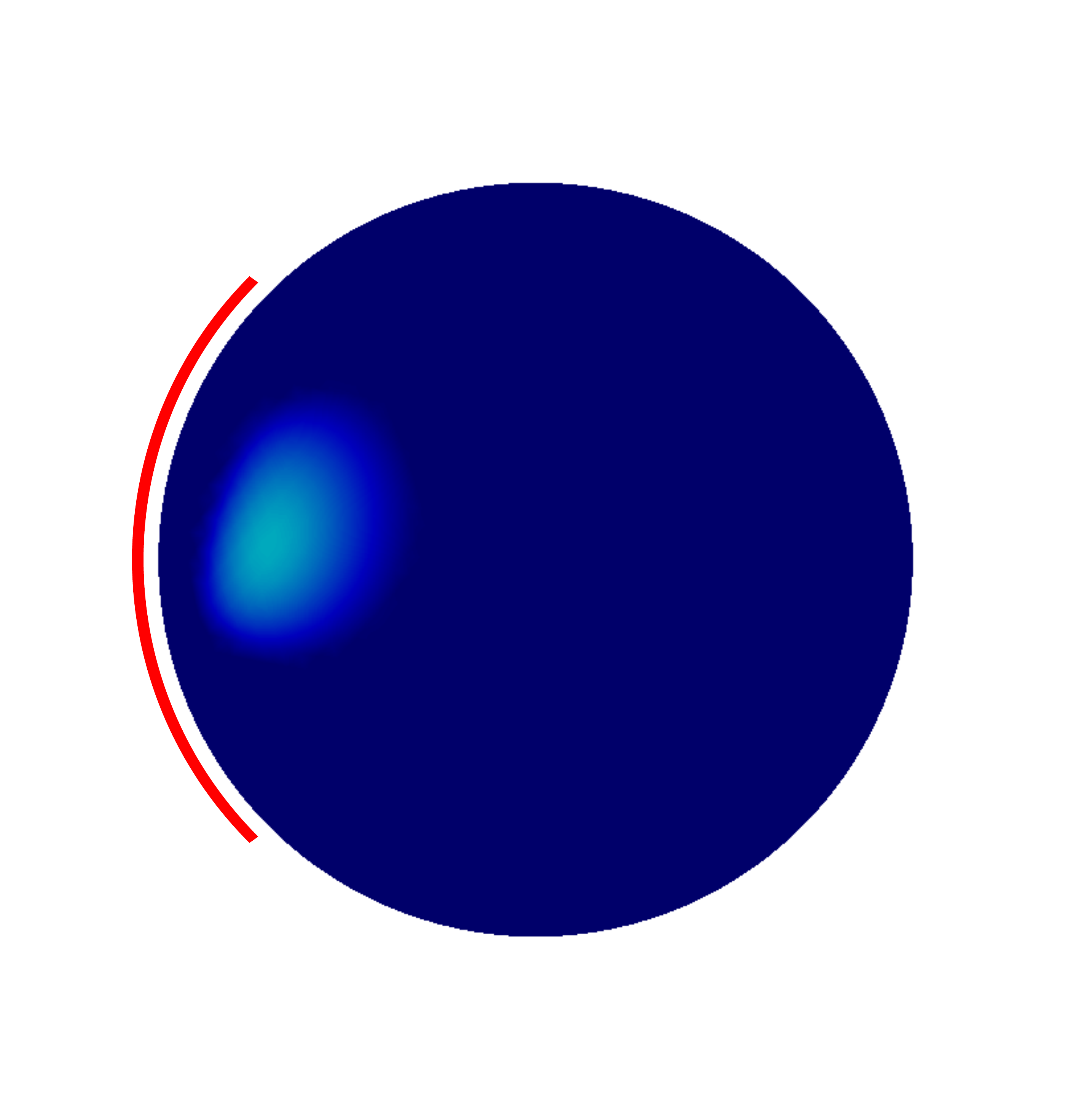}
	\caption{{}}
\end{subfigure}%
\hfill
\begin{subfigure}[b]{.25\linewidth}
\includegraphics[width = \linewidth, trim = 3.5cm 4cm 0cm 4cm, clip=true]{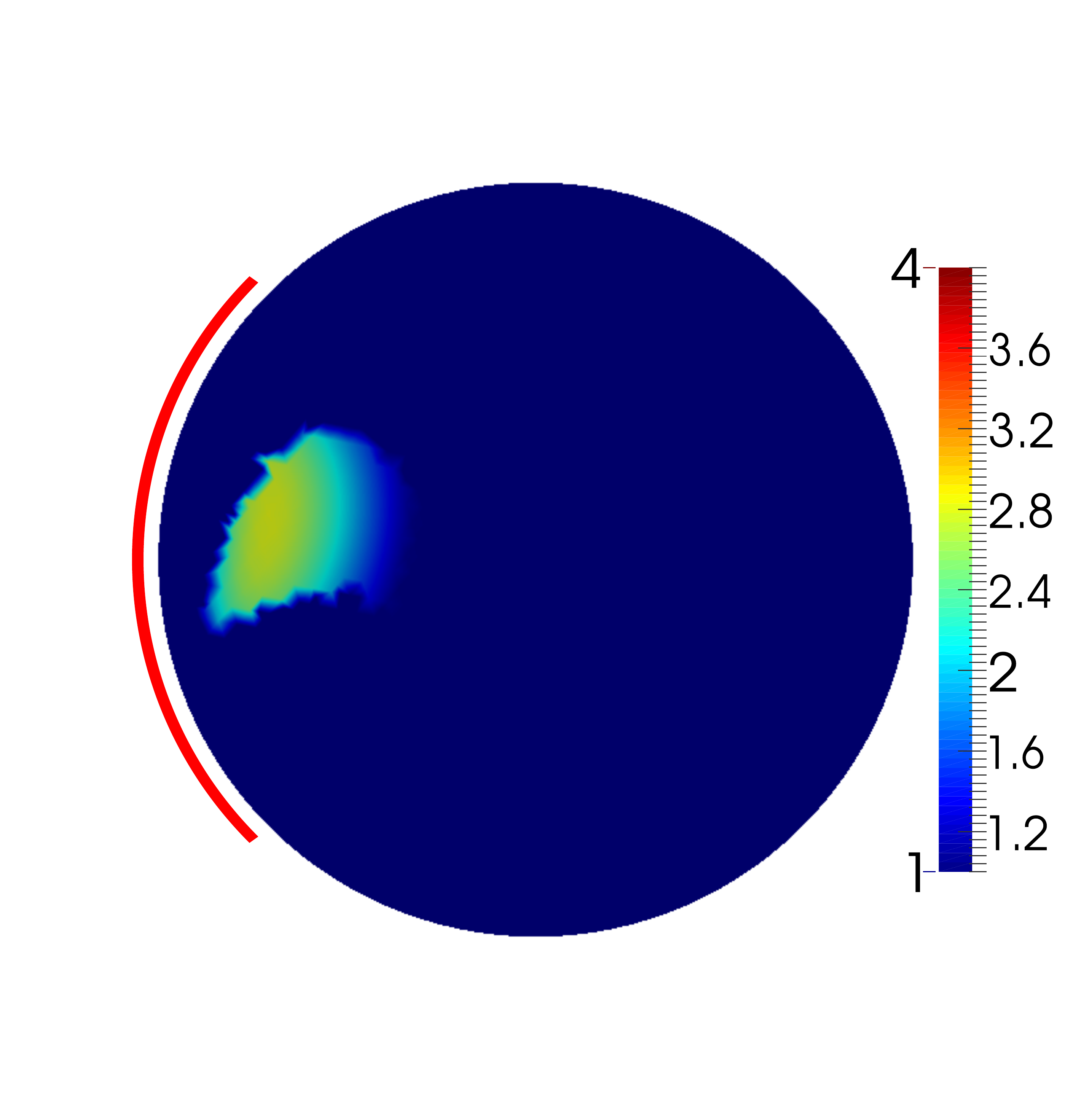}
	\caption{{}}
\end{subfigure}%
\caption{Sparse reconstruction of the phantom in \fref{fig:kitephantom}. \textbf{(a):} 50\% boundary data, no prior. \textbf{(b):} 50\% boundary data with 10\% overestimated support. \textbf{(c):} 25\% boundary data, no prior. \textbf{(d):} 25\% boundary data with 10\% overestimated support.}
\label{fig:kitepartial}
\end{figure}

In \fref{fig:circpartial} and \fref{fig:kitepartial} panels (a) and (c) it is observed that as the length of $\Gamma$ becomes smaller, the reconstructed shape of the inclusion is rapidly deformed. By including prior information about the support of $\dsr$, it is possible to rectify the deformation of the shape, and get reconstructions with almost the correct shape but with a slightly worse reconstructed contrast compared to full boundary data reconstructions. This is observed for the ball and kite-shaped inclusions in \fref{fig:circpartial} and \fref{fig:kitepartial}.

\subsection{Comparison with Total Variation Regularization} \label{sec:comparison}

Another sparsity promoting method is total variation (TV) regularization, which promotes a sparse gradient in the solution. This can be achieved by minimizing the functional
\begin{equation}
	\Psi_{\text{TV}}(\delta\gamma) \equiv \sum_{k=1}^K R_k(\delta\gamma) + P_{\text{TV}}(\delta\gamma), \enskip\delta\gamma\in\arum_0, \label{psieqTV}
\end{equation}
where the discrepancy terms $R_k$ remains the same as in \eqref{psieq}, but the penalty term is now given by
\begin{equation}
	P_{\text{TV}}(\delta\gamma) \equiv \alpha\int_{\Omega} \sqrt{\absm{\nabla\delta\gamma}^2+b}\,dx.
\end{equation}
Here $b>0$ is a constant that implies that $P_{\text{TV}}$ is differentiable, but chosen small such that $P_{\text{TV}}$ approximates $\alpha\int_{\Omega} \absm{\nabla\delta\gamma}\,dx$.

For the numerical examples, the piecewise constant phantoms in \fref{fig:circphantom} and \fref{fig:kitephantom} are used, with the same noise level as in the previous sections. The value $b = 10^{-5}$ is used for the penalty term in all the examples.

It should be noted that the color scale in the following examples is not the same scale as for the phantoms, unlike the previous reconstructions. This is because the TV reconstructions have a significantly lower contrast, in particular for the partial data reconstructions, and would be visually difficult to distinguish from the background conductivity in the correct color scale. 

\begin{figure}[h]%
\centering
\begin{subfigure}[b]{.33\linewidth}
\includegraphics[width = \linewidth, trim = 3.5cm 4cm 0cm 4cm, clip=true]{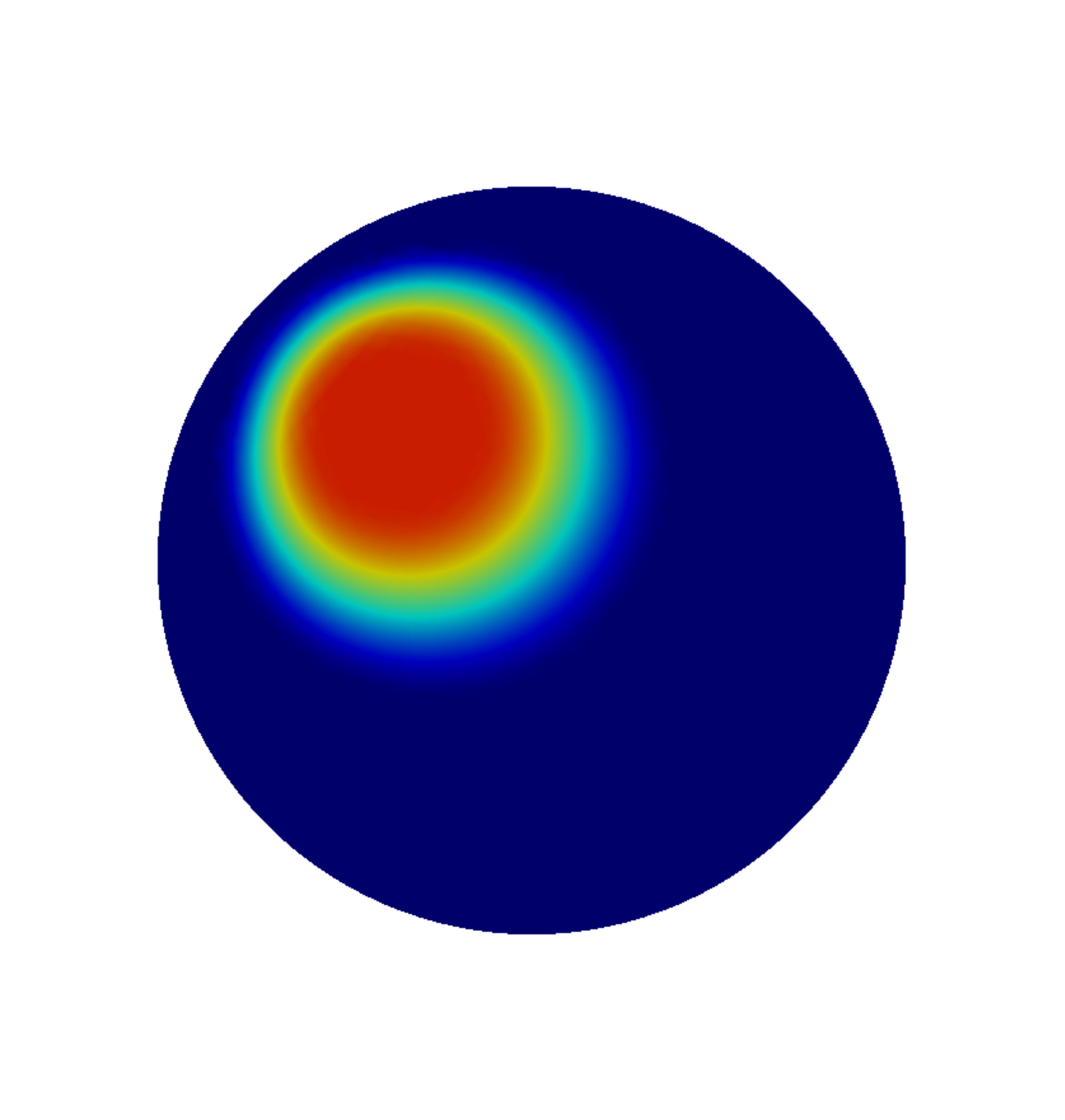}
	\caption{{}}
\end{subfigure}%
\hfill
\begin{subfigure}[b]{.33\linewidth}
\includegraphics[width = \linewidth, trim = 3.5cm 4cm 0cm 4cm, clip=true]{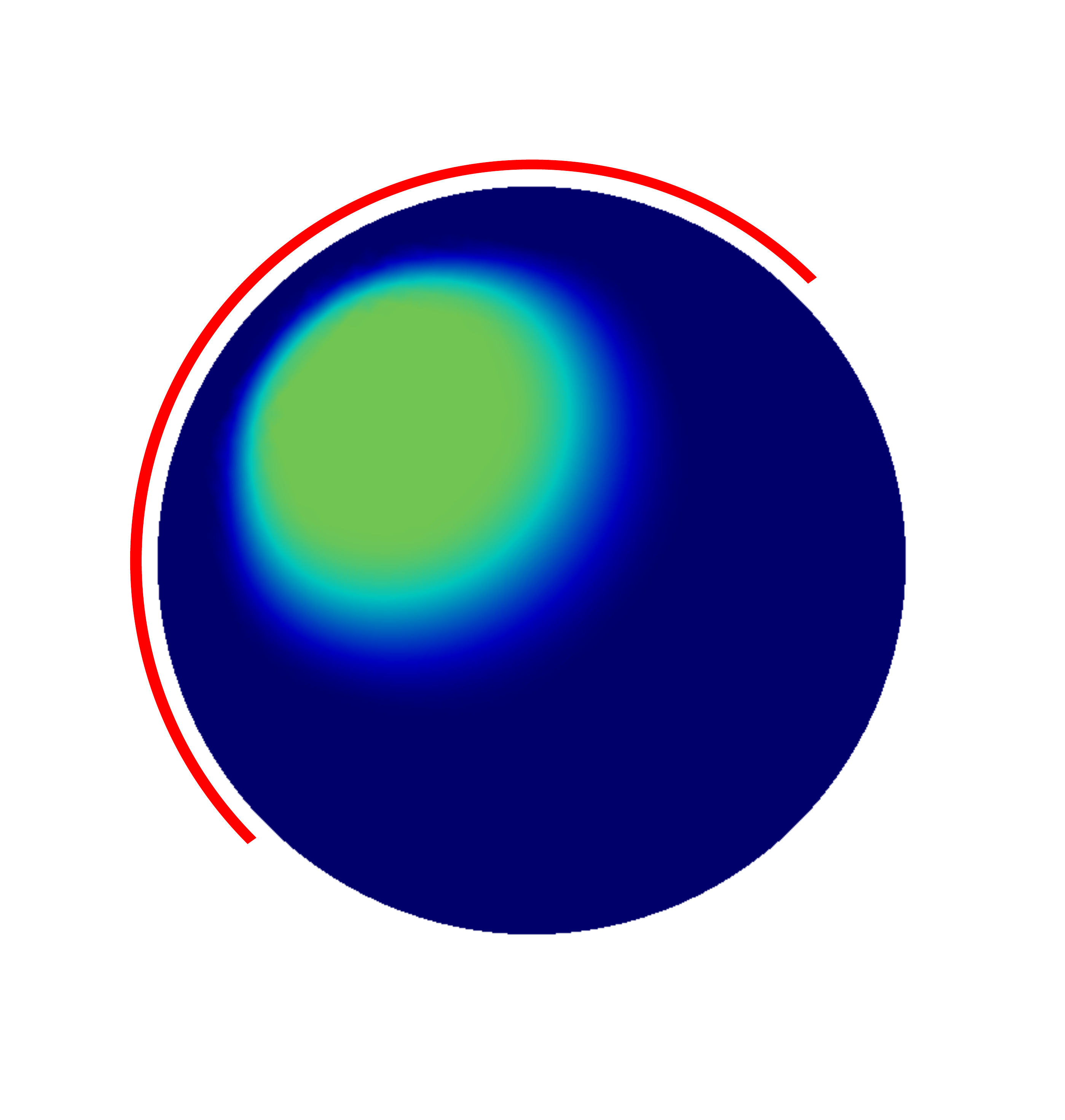}
	\caption{{}}
\end{subfigure}%
\hfill
\begin{subfigure}[b]{.33\linewidth}
\includegraphics[width = \linewidth, trim = 3.5cm 4cm 0cm 4cm, clip=true]{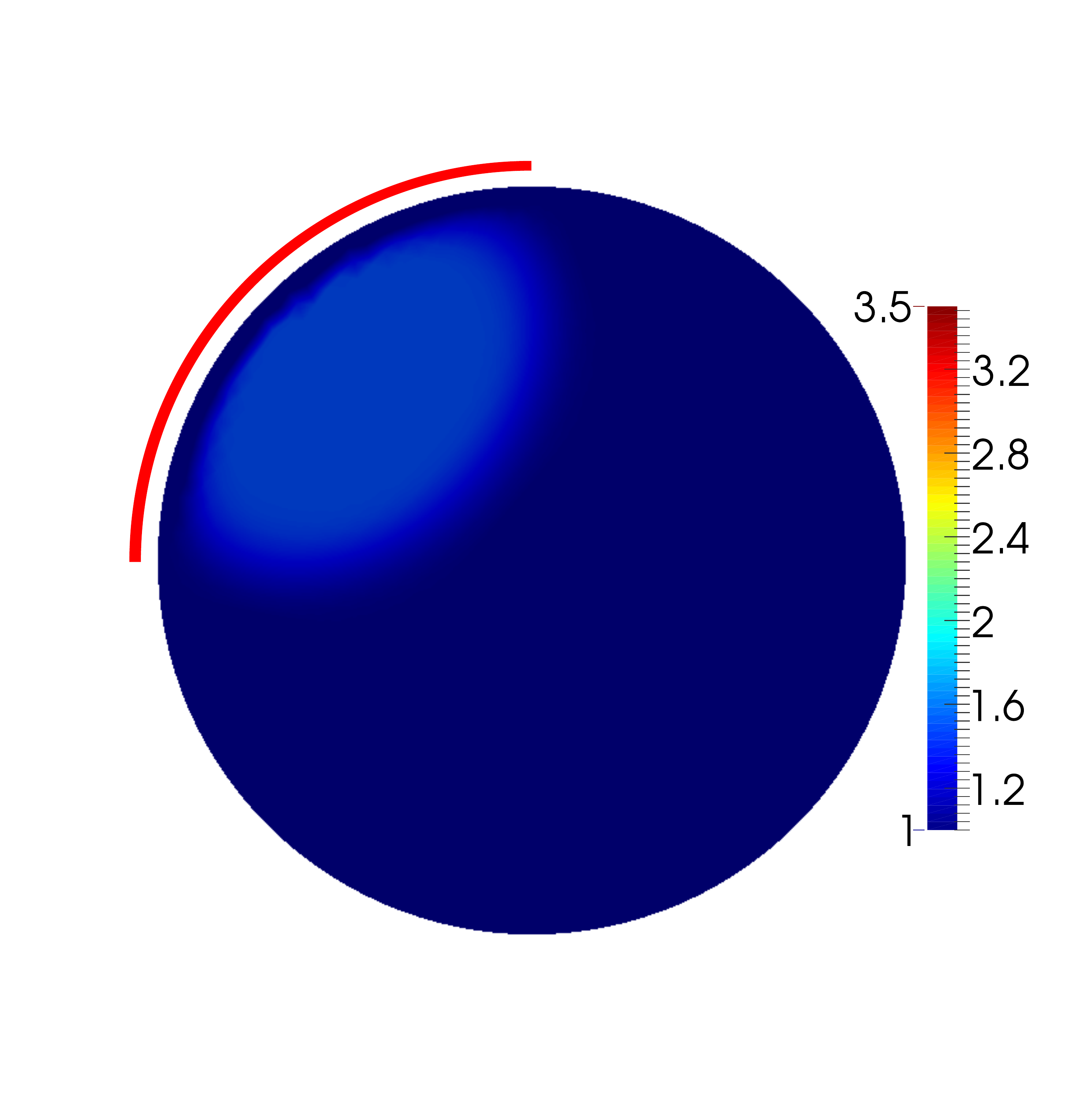}
	\caption{{}}
\end{subfigure}%
\caption{TV reconstruction of the phantom in \fref{fig:circphantom}. \textbf{(a):} Full boundary data. \textbf{(b):} 50\% boundary data. \textbf{(c):} 25\% boundary data.}
\label{fig:circpartial2}
\end{figure}

\begin{figure}[h]%
\centering
\begin{subfigure}[b]{.33\linewidth}
\includegraphics[width = \linewidth, trim = 3.5cm 4cm 0cm 4cm, clip=true]{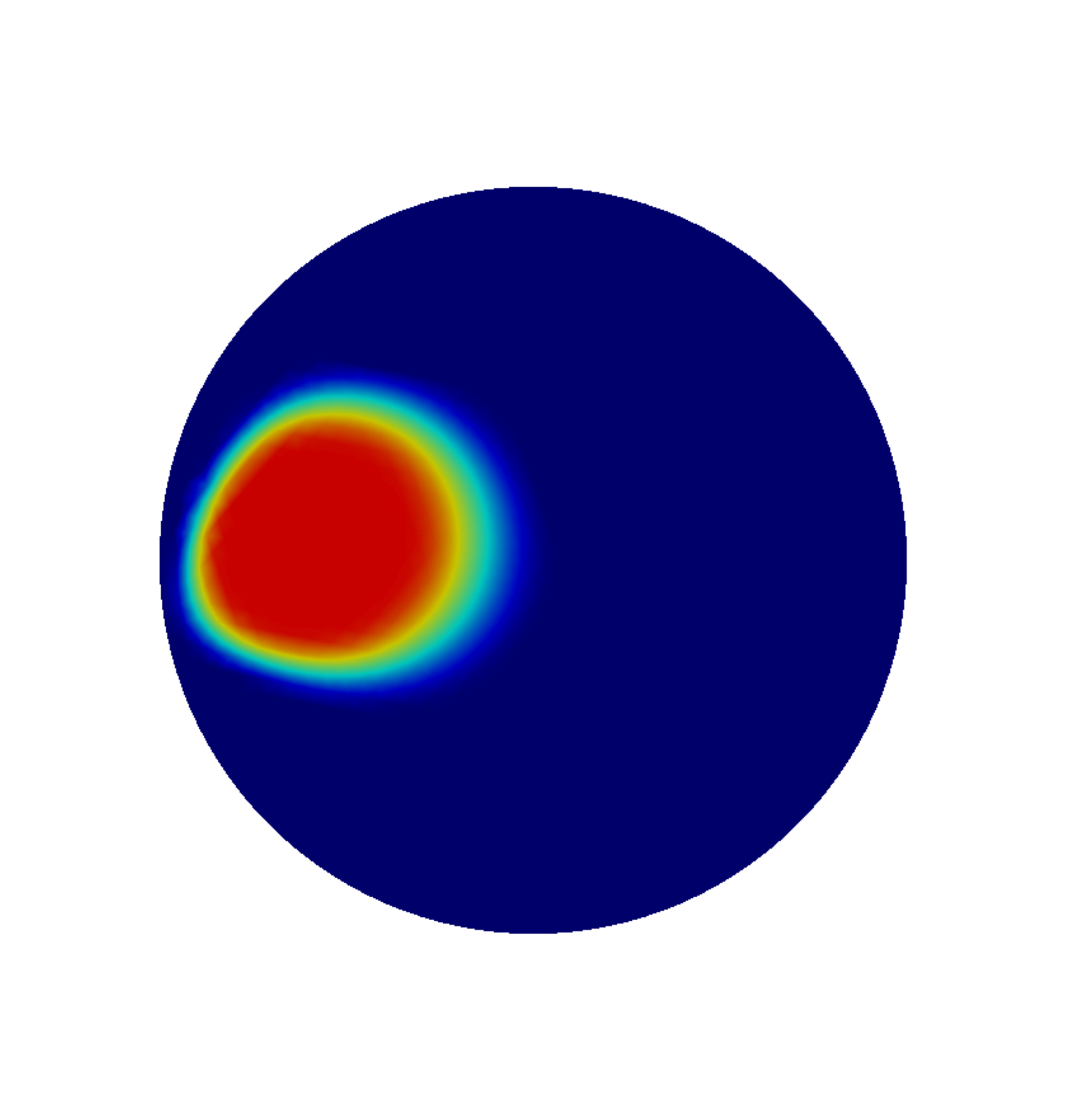}
	\caption{{}}
\end{subfigure}%
\hfill
\begin{subfigure}[b]{.33\linewidth}
\includegraphics[width = \linewidth, trim = 3.5cm 4cm 0cm 4cm, clip=true]{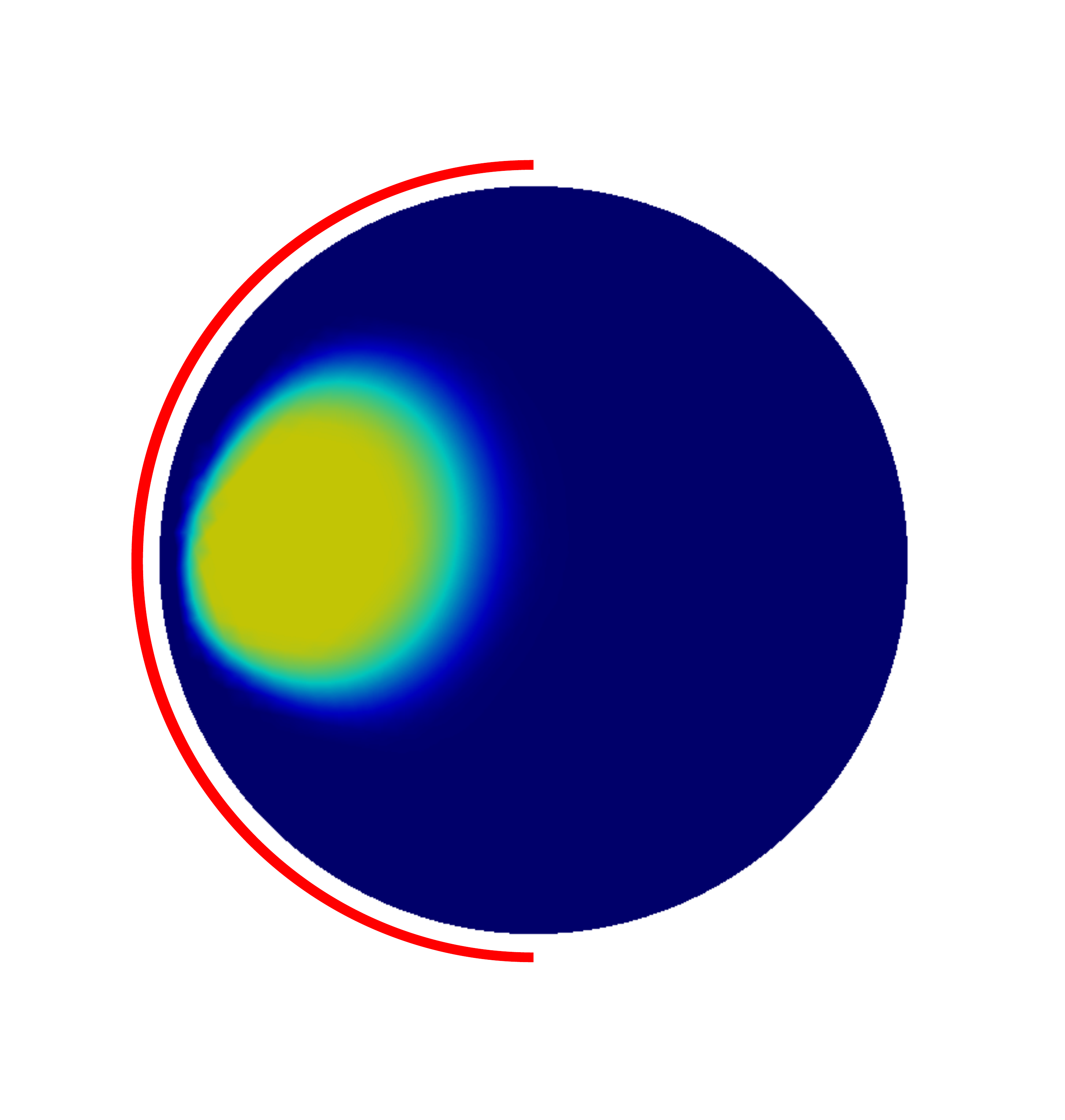}
	\caption{{}}
\end{subfigure}%
\hfill
\begin{subfigure}[b]{.33\linewidth}
\includegraphics[width = \linewidth, trim = 3.5cm 4cm 0cm 4cm, clip=true]{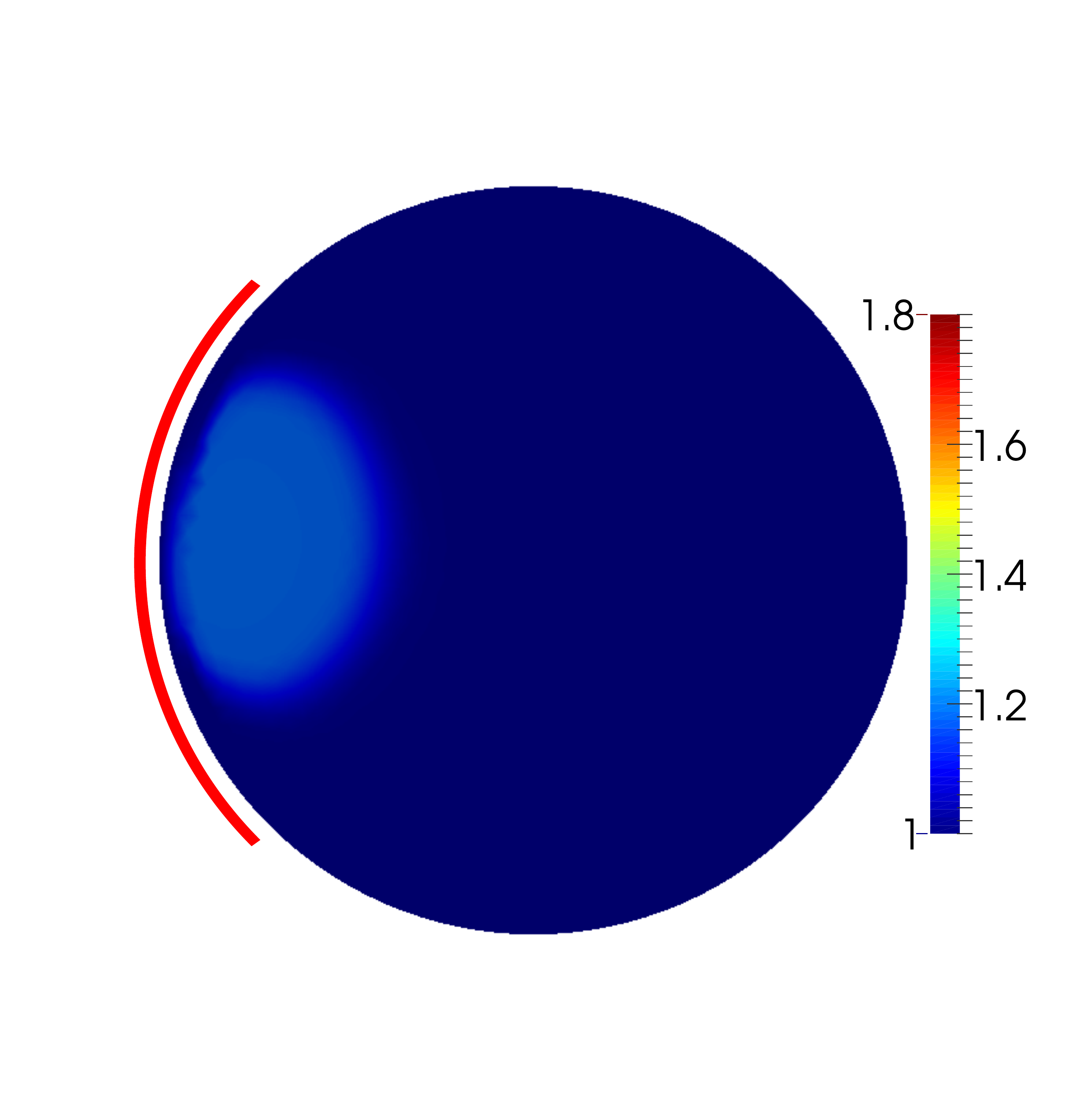}
	\caption{{}}
\end{subfigure}%
\caption{TV reconstruction of the phantom in \fref{fig:kitephantom}. \textbf{(a):} Full boundary data. \textbf{(b):} 50\% boundary data. \textbf{(c):} 25\% boundary data.}
\label{fig:kitepartial2}
\end{figure}

As seen from \fref{fig:circpartial2} and \fref{fig:kitepartial2} the support of the inclusion is slightly overestimated in the case of full boundary data, and for the partial data cases the support is slightly larger than the counterparts in \fref{fig:circpartial} and \fref{fig:kitepartial}. It is also noticed that the TV reconstructions have a much lower contrast than the $\ell_1$ sparsity reconstructions, and the contrast for the TV reconstructions is severely reduced when partial data is used. It is also observed that the same type of shape deformation occurs for both methods in case of partial data.

A typical feature of the TV regularization is piecewise constant reconstructions, however the reconstructions seen here have constant contrast levels with a smooth transition between them. There are several reason for this; and is due to the slight smoothing of the penalty term, but mostly because the discrepancy terms are not convex and may lead to local minima. The same kind of smooth transitions are also observed in TV-based methods for EIT in \cite{Borsic}. 

\section{Conclusions}\label{sec:conclusions}

We have extended the algorithm developed in \cite{jin2012}, for sparse reconstruction in electrical impedance tomography, to the case of partial data. Furthermore, we have shown how a distributed regularization parameter can be applied to utilize spatial prior information. This lead to numerical results showing improved reconstructions for the support of the inclusions and the contrast simultaneously. The use of the distributed regularization parameter enables sharper edges in the reconstruction and vastly reduces the deformation of the inclusions in the partial data problem, even when the prior is overestimated.

The algorithm can be generalized for 3D reconstruction, under further assumptions on the boundary conditions $\{g_k\}_{k=1}^K$ and the amplitude of the perturbation $\delta\sigma$. This will be considered in a forthcoming paper \cite{GardeKnudsen2014}.

\section*{Funding}

This research is supported by Advanced Grant No.\ 291405 HD-Tomo from the European Research Council.

%\nocite{*}
\bibliographystyle{gIPE}
\bibliography{artikel}

\end{document}